\documentclass[journal]{IEEEtran}

\usepackage{amsmath,graphicx}
  \usepackage{amsthm}
\usepackage{url}
\usepackage{booktabs}
\usepackage{tabularx}
\usepackage{makecell}

\usepackage{pgfplots}

\usepackage{amssymb}    
\usepackage{mathtools} 

\usepackage[caption=false,font=normalsize,labelfont=sf,textfont=sf]{subfig}

\usepackage[ruled,norelsize,lined,linesnumbered]{algorithm2e}
\usepackage{algorithmic}

\newtheorem{theorem}{Theorem}
\newtheorem{definition}{Definition}
\newtheorem{remark}{Remark}
\newtheorem{proposition}{Proposition}

\newtheorem{lemma}{Lemma}

\newtheorem{assumption}{Assumption}

\newtheorem{corollary}{Corollary}

\pgfplotsset{compat=1.18}
\usetikzlibrary{patterns.meta, arrows.meta}
\definecolor{oceanblue}{RGB}{53, 138, 190}
\definecolor{palegreen}{RGB}{179, 216, 178}
\definecolor{paleorange}{RGB}{255, 227, 178}
\definecolor{palepurple}{RGB}{242, 204, 218}
\usepackage{enumitem}

\newcommand{\R}{\mathbb{R}}

\newcommand{\B}{\mathbb{B}}
\newcommand{\Sph}{\mathbb{S}}
\newcommand{\Y}{\mathcal{Y}}

\newcommand{\norm}[1]{\left\|#1\right\|}

\newcommand{\diam}{\operatorname{diam}}
\newcommand{\conv}{\mathrm{conv}}
\newcommand{\clarke}{\mathrm{J}^c}
\newcommand{\jac}{\mathrm{J}}
\newcommand{\apjac}{\widehat{\mathrm{J}}}
\newcommand{\apnbla}{\widehat{\nabla}}

\newcommand{\BiC}{\partial^c \tilde\varphi}

\newcommand{\sav}[1]{\textcolor{red}{[#1]\raise 0.5ex \hbox{\footnotesize{SB}}}}

\usepackage{comment}

\usepackage[hidelinks]{hyperref}

\title{\LARGE \bf
Bi-ZOL: Bilevel Zeroth-Order Learning with Nonsmooth Responses
}

\author{Zhisen Jiang, Saverio Bolognani
\thanks{ 
The authors are with the Automatic Control Laboratory, ETH Zurich, 8092 Zurich, Switzerland. Email: \{zhijiang, bsaverio\}@ethz.ch.}
}

\begin{document}

\maketitle

\begin{abstract}
This paper studies lower-level-constrained bilevel optimization in a response-oracle setting, where lower-level model information is unavailable and the induced response mapping is locally Lipschitz but potentially nonsmooth. In this setting, the classical response Jacobian and reduced hypergradient may fail to exist. We propose Bilevel Zeroth-Order Learning (Bi-ZOL), a structure-guided zeroth-order method for finding stationary points of the nonsmooth reduced problem. Instead of estimating the gradient of a fully smoothed reduced hyperobjective, Bi-ZOL separates the bilevel chain-rule structure: it keeps the exact upper-level partial gradients at the queried response and uses zeroth-order sampling only to estimate the response Jacobian. This construction yields an approximate hypergradient that is more directly aligned with the Clarke chain-rule subdifferential. We show that the Bi-ZOL direction admits a partial-smoothing interpretation, quantify its pointwise structural bias, and prove finite-time convergence to a $(\delta,\epsilon)$-Bi-ZOL Frank--Wolfe stationary point. The bias is $O(\delta)$ for piecewise $C^{1,1}$ responses under local regularity and vanishes for piecewise affine responses on active-cell neighborhoods. 
Experiments on incentive-based tracking problems show that Bi-ZOL achieves smaller stationarity gaps and lower hyperobjective values than vanilla zeroth-order smoothing under comparable response-oracle budgets.
\end{abstract}


\section{Introduction}

Many control and decision systems have a hierarchical structure, where an upper-level decision maker designs a policy, incentive, or control signal, and lower-level agents respond through constrained local decisions. 
Such problems arise in traffic networks \cite{Grontas2023traffic}, energy systems \cite{Jiang2025sg}, and machine learning tasks such as hyperparameter tuning \cite{Maclaurin2015hyperparam_tuning} and Large Language Model training \cite{Reisizadeh2025LLM}. 
Mathematically, they can be modeled as bilevel optimization problems in which the lower-level problem is parameterized by the upper-level decision and is subject to local constraints. 
The key object is the response mapping from the upper-level decision to the lower-level reaction, which inherently couples the two levels.

Although various solution methods have been proposed over the years, solving such bilevel problems remains difficult due to two intertwined practical and theoretical challenges.
From a practical standpoint, standard methods rely on explicit lower-level model information, which is not available in many control settings with only response queries to the lower-level are accessible.
KKT-based \cite{Outrata2013KKT} and value-function-based methods \cite{Shen2023penalty,Ye2022BOME} convert lower-level optimality into a single-level reformulation and solve it.
Therefore, they require full access to the lower-level objective, constraints, and qualification conditions.
Hypergradient methods usually compute or estimate the response Jacobian through implicit differentiation, which requires lower-level derivatives or Hessians \cite{Zhou20261,Kornowski2024,Maljkovic2025}.
Consequently, these approaches are not directly applicable in response-oracle settings. In such scenarios, the lower level is a black box, and the exact response Jacobian which represents the lower-level sensitivity information required by classical hypergradient methods cannot be computed from simple oracle queries.

Compounding this practical hurdle is a fundamental theoretical challenge: lower-level constraints inherently introduce nonsmoothness.
In this case, the response Jacobian required by hypergradient methods might even fail to exist.
When the lower-level problem contains inequality constraints, small changes in the upper-level variable may change the active constraint set of the lower-level solution. 
Thus, the response mapping is generally only locally Lipschitz, and it may fail to be continuously differentiable even when all objective functions are smooth. 
The lower-level constrained bilevel problem should be treated, in general, as a nonsmooth and nonconvex optimization task, where standard bilevel methods based on smooth assumptions do not apply directly.

Existing studies address this nonsmoothness issue mainly in two ways, each with different limitations.
The first line of works replaces the original nonsmooth response relation by a smooth auxiliary problem.
Barrier-based methods move the lower-level constraints into a barrier \cite{Jiang2024barrier}, or a smoothing-barrier augmented Lagrangian term \cite{Xu2024barrier_smoothing}, so that the algorithm targets a smooth parametric problem.
Similarly, Gap-function methods replace the optimal response by a smooth regularized primal-dual optimality residual \cite{Yao2024regularized_gap_function}.
Some perturbation-based methods add random linear perturbations to the lower-level objective so that the perturbed response is differentiable almost surely for linearly constrained lower-level problems \cite{Khanduri2023,Khanduri2025}.
These methods enable gradient-based algorithms to solve smooth approximations, but their connection to the original bilevel problem relies on exactness, consistency, or stationarity transfer analysis, often under additional structural assumptions.

The second line of works directly uses tools of nonsmooth optimization to solve the original problem.
For example, BIG Hype \cite{Grontas2024BIGHype} computes generalized hypergradients through conservative-Jacobian-based sensitivity learning and applies projected generalized-gradient updates. 
This framework gives a rigorous nonsmooth treatment, but it relies on learned generalized sensitivity information and uses diminishing stepsizes to handle the oscillations caused by the subgradient method.

To handle the black-box nature and nonsmoothness simultaneously, a natural strategy is to apply zeroth-order smoothing to the reduced hyperobjective \cite{Marco2026ZOBA, Maheshwari2024cdc, Jia2026}. 
Specifically, such methods use response queries to evaluate the upper-level objective as a black-box function, and then apply finite differences to estimate the gradient of this smoothed surrogate.
We refer to this strategy as vanilla zeroth-order smoothing (VZO), or full smoothing, because the gradient estimator is applied to the entire hyperobjective.
This approach avoids both lower-level model access and the undefined response Jacobian, but it neglects the bilevel information structure considered here: the upper-level has the closed-form expression of its objective and partial derivatives, while only the response mapping and its sensitivity are unavailable. 
Thus, the full-smoothing method may shift the stationary structure and introduce avoidable bias. 
Although recent works connect full-smoothing solutions to nonsmooth stationary points of general nonsmooth optimization \cite{Masiha2026, Kornowski2024, Cui2022MPEC}, their analyses are not tailored to the bilevel chain-rule structure. 
This motivates a partial-smoothing strategy that regularizes only the unavailable and possibly undefined response Jacobian.

To address these practical and theoretical challenges, we propose Bilevel Zeroth-Order Learning (Bi-ZOL) for lower-level-constrained bilevel optimization under a response-oracle setting.
The key idea is to smooth only the response mapping that appears in the missing chain-rule term.
Thus, Bi-ZOL keeps the upper-level partial derivatives at the true observed response while using randomized response queries to estimate the smoothed response sensitivity.
Our contributions are three-fold:
    \begin{itemize}
        \item We provide a systematic first-order characterization of the nonsmooth geometry induced by lower-level constraints and use it to design a structure-guided approximate hypergradient. 
Under a locally Lipschitz response mapping, we characterize the reduced hyperobjective through a bilevel Clarke chain-rule structure, which identifies the object that replaces the classical hypergradient when the response Jacobian is undefined. 
Motivated by this structure, we construct a partial-smoothing approximate hypergradient that keeps the upper-level partial derivatives evaluated at the true response and applies zeroth-order smoothing only to the unavailable response sensitivity.
This construction avoids estimating the gradient of a fully smoothed reduced hyperobjective and gives an approximate hypergradient that is directly tied to the nonsmooth bilevel chain-rule structure.

\item We introduce and analyze a Bi-ZOL Frank--Wolfe stationarity certificate for the proposed approximate hypergradient, and establish its connection to Clarke Frank--Wolfe stationarity of the original nonsmooth reduced problem.
We quantify the pointwise structural bias between the Bi-ZOL direction and the bilevel Clarke chain-rule subdifferential, and show how this bias controls the stationarity transfer from Bi-ZOL Frank--Wolfe stationary point to Clarke Frank--Wolfe stationary point. 
This analysis further identifies response structures under which the transfer becomes sharper: the bias is $O(\delta)$ for piecewise $C^{1,1}$ responses under local regularity, and it vanishes for piecewise affine responses.

\item We propose Bi-ZOL algorithm, a response-oracle zeroth-order Frank--Wolfe method for achieving the proposed stationarity certificate. 
Bi-ZOL uses randomized two-point response queries to estimate the proposed approximate hypergradient. 
With a Frank--Wolfe update, we prove finite-time convergence to a $(\delta,\epsilon)$-Bi-ZOL Frank--Wolfe stationary point and establish the corresponding response-oracle complexity. 
    \end{itemize}

The remainder of this paper is organized as follows. 
Section~\ref{sec:problem} introduces the problem formulation and the response-oracle setting.
Section~\ref{sec:first_order_geometry} studies the nonsmooth bilevel first-order geometry, constructs the partial-smoothing approximate hypergradient, and analyzes the corresponding structural bias and stationarity transfer. 
Section~\ref{sec:solution_methods} presents the Bi-ZOL algorithm, its zeroth-order response-sensitivity estimator, and the finite-time convergence and response-oracle complexity results. 
Section~\ref{sec:simulation_results} reports numerical results on incentive-based tracking problems. 
Section~\ref{sec:conclusions} concludes the paper.

\subsection{Notation and Preliminaries}
\label{subsec:NP}
\subsubsection*{Notation}
We denote the unit closed ball centered at the origin by $\B_n \coloneqq \{\xi \in \R^n : \|\xi\| \le 1\}$, and the unit sphere by $\Sph^{n-1} \coloneqq \{w \in \R^n : \|w\| = 1\}$. 
For a set $\mathcal{S}$, $\conv(\mathcal{S})$ denotes its convex hull, and the distance from a point $z$ to $\mathcal{S}$ is defined as $\operatorname{dist}(z, \mathcal{S}) \coloneqq \inf_{s \in \mathcal{S}} \|z - s\|$. 
Let $\mathrm{Unif}(\mathcal{S})$ denote the uniform distribution over a set $\mathcal{S}$. 
For a differentiable vector-valued mapping $F: \R^n \to \R^m$, its Jacobian matrix is denoted by $\jac F(x) \in \R^{m \times n}$.
For complexity analysis, we use standard Big-O notation, where $O(\cdot)$ and $\Theta(\cdot)$ hide absolute numerical constants that are independent of problem dimensions and structural parameters.

\subsubsection*{Nonsmooth analysis}

Let $\mathcal O\subseteq\mathbb R^n$ be an open set. 
A function $f:\mathcal O\to\mathbb R$ is called locally Lipschitz on
$\mathcal O$ if, for every $x\in\mathcal O$, there exist a radius
$r_x>0$ and a constant $L_x>0$ such that
$B_{r_x}(x)\subseteq\mathcal O$ and
\(
|f(u)-f(v)|
\le
L_x\|u-v\|,
\,
\forall u,v\in B_{r_x}(x).
\)
If the same constant $L$ works for all $u,v$ in a given set, then $f$ is
called $L$-Lipschitz on that set. In particular, an $L$-Lipschitz
function on an open set is locally Lipschitz on that set.

By Rademacher's theorem, locally Lipschitz functions are differentiable
almost everywhere in the sense of Lebesgue measure. Hence, for any locally
Lipschitz function $f:\mathcal O\to\mathbb R$ and any point
$x\in\mathcal O$, the Clarke subdifferential is defined as
\(
\partial^c f(x)
:=
\conv
\left\{
\lim_{k\to\infty}\nabla f(x^k)
\ \middle|\
x^k\to x,\ x^k\in\Omega_f
\right\},
\)
where $\Omega_f\subseteq\mathcal O$ is the full-measure set of points at
which $f$ is differentiable \cite{Clarke1990}. Equivalently,
$\partial^c f(x)$ is the convex hull of all limit points of
$\nabla f(x^k)$ over sequences of differentiable points converging to
$x$. If $f$ is continuously differentiable at $x$, then
\(
\partial^c f(x)=\{\nabla f(x)\}.
\)

Given $\delta>0$, the Goldstein $\delta$-subdifferential of $f$ at $x$ is
defined as
\(
\partial_\delta f(x)
:=
\conv
\left(
\bigcup_{z\in B_\delta(x)\cap\mathcal O}
\partial^c f(z)
\right)
\)
in \cite{Goldstein1977}.
That is, $\partial_\delta f(x)$ collects all convex combinations of
Clarke subdifferentials at points in a $\delta$-neighborhood of $x$. 

We also need the generalized derivative of vector-valued mappings. Let
$F:\mathcal O\to\mathbb R^m$ be locally Lipschitz on the open set
$\mathcal O$. By Rademacher's theorem, $F$ is differentiable almost
everywhere on $\mathcal O$. Its \textit{Clarke generalized Jacobian} is
the set-valued mapping
\(
\clarke F:\mathcal O\rightrightarrows\mathbb R^{m\times n}
\)
defined by
\(
\clarke F(x)
:=
\conv
\left\{
\lim_{k\to\infty}\jac F(x^k)
\ \middle|\
x^k\to x,\ x^k\in\Omega_F
\right\},
\)
where $\Omega_F\subseteq\mathcal O$ is the full-measure set of points at
which $F$ is differentiable \cite{Clarke1990}. Thus, $\clarke F(x)$ is
the convex hull of all limiting Jacobians of $F$ around $x$. When $F$ is
continuously differentiable at $x$, this set reduces to the singleton
\(
\clarke F(x)=\{\jac F(x)\}.
\)

\section{Problem Formulation}
\label{sec:problem}
In this work, we study a bilevel problem with constrained lower-level problem:
\begin{alignat}{2}
    & \min_{x,y} \quad && \varphi \left( x, y \right) \label{model:p1} \tag{${\rm P}1$}\\
    & \text{subject to:} \quad &&  x \in \mathcal{X},  \\
    & && y \in S(x) \coloneqq \arg\min_{\xi\in\Y} \, g(x,\xi),
\end{alignat}
where $x \in \R^n, y \in \R^m$ denote the decisions of the upper and lower level, respectively. 
$S(x)$ denotes the set of optimal solutions for the lower-level problem.

The following assumption turns the lower-level solution set into a well-defined response mapping and gives the regularity needed for our partial smoothing method. 

\begin{assumption}[Lower-level response regularity]
\label{ass:response}
The lower-level solution mapping \(S(x)\) is single-valued on an open
neighborhood \(\mathcal N\) of \(\mathcal X\). We write its unique value as
\(y(x)\) \footnote{By a slight abuse of notation, $y$ denotes both the response mapping $y(\cdot)$ and its value $y(x)$.}.

The response mapping \(y:\mathcal N\to\mathbb R^m\) is locally Lipschitz on
\(\mathcal N\). Since \(\mathcal X\) is compact and
\(\mathcal N\) is an open neighborhood of \(\mathcal X\), there exists
\(\bar\delta>0\) such that
\(
\mathcal X+\bar\delta\mathbb B_n\subseteq \mathcal N .
\)
We fix such a \(\bar\delta\). On the compact inflated set
\(
\mathcal X_{\bar\delta}
\coloneqq
\mathcal X+\bar\delta\mathbb B_n,
\)
the response map is bounded and \(L_y\)-Lipschitz, namely,
\[
\|y(x)-y(x')\|
\le
L_y\|x-x'\|,
\qquad
\forall x,x'\in\mathcal X_{\bar\delta},
\]
and
\(
\sup_{u\in\mathcal X_{\bar\delta}}\|y(u)\|
\le
Y_{\max}.
\)
All smoothing radius used in the Bi-ZOL construction satisfy
\(0<\delta\le\bar\delta\).
\end{assumption}

\begin{remark}[A sufficient condition for response Lipschitz continuity]
\label{rem:response_lipschitz}
Assumption~\ref{ass:response} is satisfied by standard parameterized strongly convex
lower-level problems whose objective and feasible set are well-defined for
all parameters in an open neighborhood of \(\mathcal X\).
Suppose that there exists an open neighborhood $\mathcal N$ of $\mathcal X$
such that the lower-level problem
\[
S(x):=\arg\min_{\xi\in\mathcal Y} g(x,\xi)
\]
is well-defined for every $x\in\mathcal N$. Assume that $\mathcal Y$ is closed
and convex, $g(x,\cdot)$ is $\mu$-strongly convex on $\mathcal Y$ uniformly in
$x\in\mathcal N$, and $\nabla_y g(x,\xi)$ is Lipschitz continuous in $x$
uniformly in $\xi\in\mathcal Y$, i.e.,
\[
\|\nabla_y g(x,\xi)-\nabla_y g(x',\xi)\|
\le
L_{gx}\|x-x'\|,
\,
\forall x,x'\in\mathcal N,\, \forall \xi\in\mathcal Y.
\]
Then $S(x)$ is a singleton for every $x\in\mathcal N$. Writing
$y(x)$ for its unique element, the response mapping
$y:\mathcal N\to\mathbb R^m$ satisfies
\[
\|y(x)-y(x')\|
\le
\frac{L_{gx}}{\mu}\|x-x'\|,
\qquad
\forall x,x'\in\mathcal N.
\]
Hence, $y(\cdot)$ is locally Lipschitz on $\mathcal N$.
\end{remark}

The next assumptions on the upper level are standard:
\begin{assumption}[Feasible set]
\label{ass:feasible}
$\mathcal{X}\subset\R^n$ is nonempty, convex, and compact.
Let $D\coloneqq \diam(\mathcal{X})=\sup_{x,x'\in\mathcal{X}}\norm{x-x'}<\infty$.
\end{assumption}

\begin{assumption}[Upper-level regularity]
\label{ass:upper_regular}
The upper-level objective satisfies
\(
\varphi\in C^1(\mathbb R^n\times\mathbb R^m).
\)
Moreover, there exist finite constants
$L_{1x},L_{1y},L_{2x},L_{2y},M_y>0$ such that, for all relevant
$(x,y),(x',y')$ evaluated,
\[
\left\|
\nabla_1\varphi(x,y)-\nabla_1\varphi(x',y')
\right\|
\le
L_{1x}\|x-x'\|+L_{1y}\|y-y'\|,
\]
\[
\left\|
\nabla_2\varphi(x,y)-\nabla_2\varphi(x',y')
\right\|
\le
L_{2x}\|x-x'\|+L_{2y}\|y-y'\|,
\]
and
\(
\left\|
\nabla_2\varphi(x,y)
\right\|
\le
M_y.
\)
\end{assumption}

\section{Nonsmooth bilevel first-order geometry}
\label{sec:first_order_geometry}
With Assumption \ref{ass:response}, we substitute the lower-level response $y(x)$ in $\varphi$, which equivalently express \eqref{model:p1} as the following reduced problem:
\begin{equation}
    \underset{x \in \mathcal{X}}{\text{minimize}} \quad  \varphi \left( x,y\left( x \right) \right) \eqqcolon \tilde{\varphi} \left( x \right) \label{model:p2},  \tag{${\rm P}2$}
\end{equation}
and we denote $\tilde{\varphi} \left( x \right)$ as the hyperobjective.

\subsection{Approximate Hypergradient of Lower-level Constrained Bilevel Problem via Randomized Smoothing}
\label{sec:approx_hypergrad}
If the response mapping $y(\cdot)$ is differentiable at $x$, then the
classical chain rule gives the hypergradient
\begin{equation}
\label{eq:classical_hypergradient}
\nabla \tilde\varphi(x)
=
\nabla_1\varphi(x,y(x))
+
\jac y(x)^\top \nabla_2\varphi(x,y(x)).
\end{equation}
However, in constrained lower-level problems, the response mapping
$y(\cdot)$ is generally nonsmooth. Hence $\jac y(x)$ may fail to exist, and
the classical hypergradient in \eqref{eq:classical_hypergradient} is not
well-defined at nonsmooth points. We therefore describe the first-order
geometry of the reduced hyperobjective through the Clarke subdifferential.

For notational convenience, throughout this paper we write
\[
d_x(x)
\coloneqq
\nabla_1\varphi(x,y(x)),
\qquad
d_y(x)
\coloneqq
\nabla_2\varphi(x,y(x)).
\]
The following result gives the bilevel chain-rule representation of the
Clarke subdifferential of the reduced hyperobjective.

\begin{proposition}[Bilevel chain-rule Clarke subdifferential]
\label{prop:bl_chain_rule}
Under Assumptions~\ref{ass:response} and~\ref{ass:upper_regular}, for every
$x\in\mathcal X$,
\begin{equation}
\label{eq:clarke_subdiff_bl_chain_rule}
\partial^c \tilde\varphi(x)
=
\left\{
d_x(x)+V^\top d_y(x)
:
V\in \clarke y(x)
\right\}.
\end{equation}
\end{proposition}

\begin{proof}
The proof is given in Appendix~\ref{app:proof_bl_chain_rule}.
\end{proof}

Proposition~\ref{prop:bl_chain_rule} shows that the nonsmooth first-order
object of the bilevel hyperobjective is generally a set rather than a single
vector. When $y(\cdot)$ is continuously differentiable around $x$, the
Clarke generalized Jacobian reduces to
\(
\clarke y(x)=\{\jac y(x)\},
\)
and $\partial^c\tilde\varphi(x)$ reduces to the classical hypergradient in
\eqref{eq:classical_hypergradient}. At nonsmooth points,
$\partial^c\tilde\varphi(x)$ collects all limiting chain-rule directions
induced by the lower-level response mapping.

Although Proposition~\ref{prop:bl_chain_rule} characterizes the nonsmooth first-order geometry of the reduced hyperobjective via the Clarke generalized Jacobian $\clarke y(x)$, a practical challenge arises because the upper level only observes the response values $y(x)$ and lacks the means to precisely construct this set.
Therefore, the key task is to construct a computable direction that respects the chain-rule structure in Proposition~\ref{prop:bl_chain_rule}.

Since the core difficulty lies in the undefined response Jacobian, our strategy is to apply targeted smoothing strictly to this sensitivity term to obtain an approximate Jacobian, ultimately leading to a tractable approximate hypergradient.
Let $\bar\delta>0$ be such that
\(
\mathcal X+\bar\delta\mathbb B_n\subseteq \mathcal N,
\)
where $\mathcal N$ is the open neighborhood in
Assumption~\ref{ass:response}. For any $0<\delta\le\bar\delta$, we define the Bi-ZOL approxiamte hypergradient as
\begin{equation}
\label{eq:approx_hypergradient}
g_\delta(x)
=
\nabla_1\varphi(x,y(x))
+
\jac y_\delta(x)^\top
\nabla_2\varphi(x,y(x)),
\end{equation}
where
$y_\delta(x)
\coloneqq
\mathbb E_{\xi\sim{\rm Unif}(\mathbb B_n)}
\left[
y(x+\delta \xi)
\right]$.
Since $y(\cdot)$ is locally Lipschitz on $\mathcal N$, it is Lipschitz on the
compact query region $\mathcal X+\bar\delta\mathbb B_n$.
Moreover, $y_\delta(\cdot)$ is smooth.

The definition in \eqref{eq:approx_hypergradient} follows a
``keep exact information and surrogate only the missing sensitivity'' principle.
The upper-level partial derivatives
\(
\nabla_1\varphi(x,y(x)),
\,
\nabla_2\varphi(x,y(x))
\)
are evaluated at the actual lower-level response $y(x)$ and are not
twisted. 
The only changing object is the response sensitivity, where the
undefined Jacobian $\jac y(x)$ is replaced by the Jacobian of the smoothed
response $\jac y_\delta(x)$. 

\begin{remark} [Comparison to full smoothing]
This construction differs from applying smoothing directly to
the scalar hyperobjective $\tilde\varphi(x)$.
A full smoothing approach defines a smooth surrogate via the expectation by 
\(
\tilde\varphi_\delta(x) \coloneqq \mathbb{E}_{\xi \sim {\rm Unif}(\mathbb{B}_n)} \left[ \tilde\varphi(x + \delta \xi) \right],
\)
whose exact gradient is given by $\nabla \tilde\varphi_\delta(x)$.
As illustrated by the gradient formula, full smoothing averages not only the nonsmooth response sensitivity $\jac y$, but also convolutions of the upper-level partial derivatives $\nabla_1 \varphi$ and $\nabla_2 \varphi$ over a $\delta$-neighborhood. This unselective averaging inadvertently distorts the ground-truth landscape of the explicitly known upper-level objective and injects substantial unnecessary variance during iterations. In sharp contrast, the Bi-ZOL direction $g_\delta(x)$ in \eqref{eq:approx_hypergradient} strictly preserves the exact upper-level first-order information at the actual response $y(x)$, isolating the randomized smoothing exclusively to the undefined response sensitivity term.
\end{remark}

\subsection{Anchored Chain-Rule Model}
\label{sec:anchored_chain_rule_model}

We next give a model-based interpretation of
\eqref{eq:approx_hypergradient}. For a fixed anchor point $x\in\mathcal X$,
define the anchored chain-rule model
\begin{equation}
\label{eq:anchored_chain_rule_model}
Q_x(u)
\coloneqq
\varphi(x,y(x))
+
\left\langle d_x(x),u-x\right\rangle
+
\left\langle d_y(x),y(u)-y(x)\right\rangle .
\end{equation}
Here, $u$ serves as the local model variable around the anchor $x$. Essentially, $Q_x(u)$ represents a local linearization of the upper-level objective $\varphi$ at the decision pair $(x, y(x))$. 
By freezing the upper-level partial derivatives while explicitly retaining the lower-level response map $y(u)$, this anchored model remains nonsmooth, yet captures the exact first-order geometry.

The following result shows that the Bi-ZOL direction \eqref{eq:approx_hypergradient} is the gradient of a
partially smoothed version of $Q_x$ and, at the same time, a Goldstein
subgradient of the same anchored model.

\begin{figure}[htbp]
    \centering
    \includegraphics[width=\columnwidth]{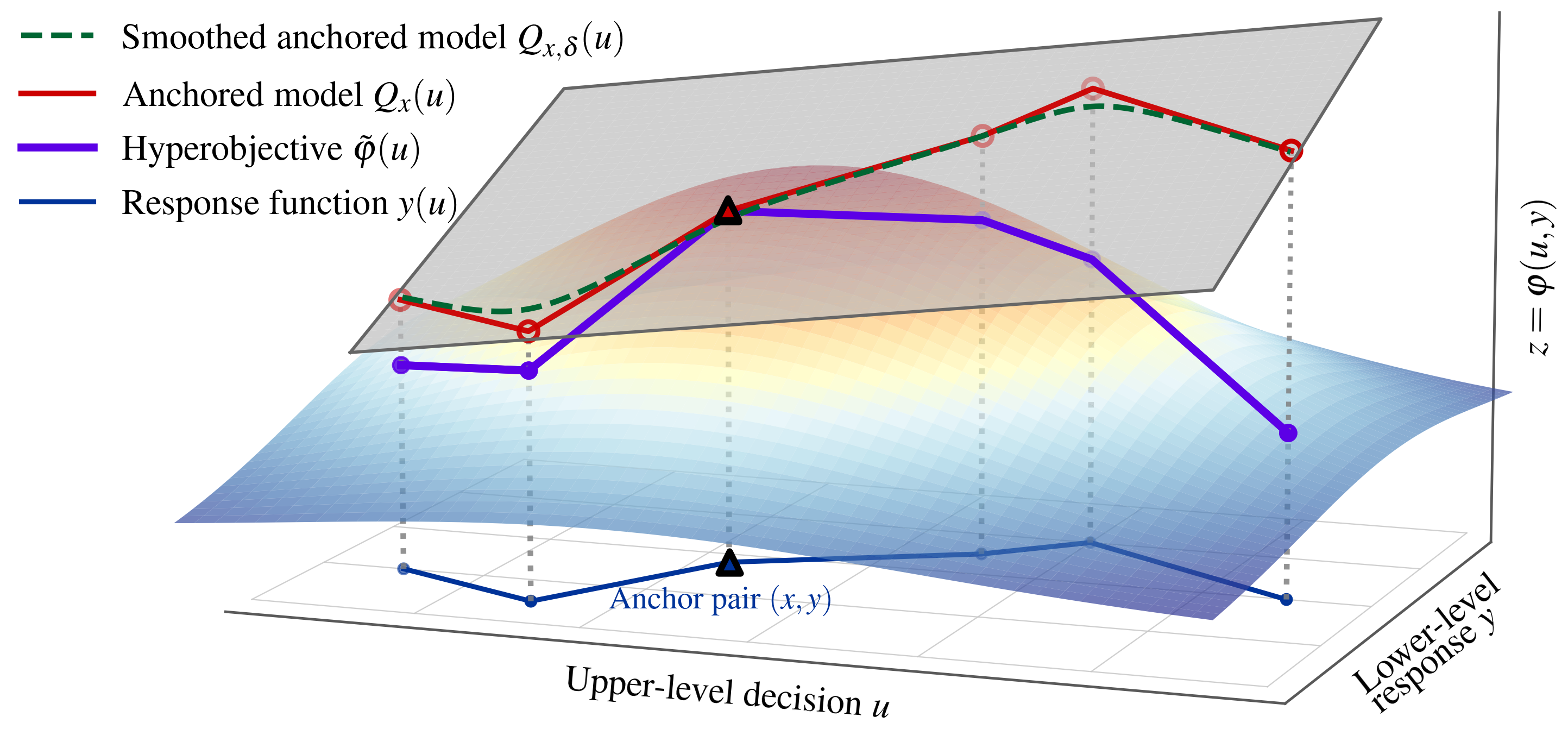}
    \caption{Geometric interpretation of the anchored chain-rule model at a fixed anchor $x$.
The lower-level response $y(u)$ is shown on the ground decision plane (blue).
Lifting $u\mapsto (u,y(u))$ yields the hyperobjective $\tilde\varphi(u)=\varphi(u,y(u))$ on the upper-level surface (purple).
At the anchor $(x,y(x))$, the tangent plane of $\varphi$ defines the linear model $T_x(u,y)$.
Restricting $y$ to the piecewise-linear response gives the nonsmooth anchored model $Q_x(u)$ (red), which preserves the Clarke first-order geometry of $\tilde\varphi$ at $u=x$ (Proposition~\ref{prop:anchored_model_interpretation}).
Replacing $y(u)$ by the Bi-ZOL smoothed response $y_\delta(u)$ on the same plane produces $Q_{x,\delta}(u)$ (green dashed), whose gradient at the anchor is exactly $g_\delta(x)$.}
    \label{fig:anchored_model_interpretation}
\end{figure}

\begin{proposition}[Anchored model interpretation]
\label{prop:anchored_model_interpretation}
Under Assumptions~\ref{ass:response} and~\ref{ass:upper_regular}, for every
$x\in\mathcal X$ and $0<\delta\le\bar\delta$, the following statements hold.

First,
\begin{equation}
\label{eq:anchored_model_same_clarke}
\partial^c Q_x(x)=\partial^c\tilde\varphi(x).
\end{equation}

Second, define the partially smoothed anchored model
\begin{equation*}
\begin{aligned}
Q_{x,\delta}(u)
\coloneqq
\varphi(x,y(x))
&+
\left\langle d_x(x),u-x\right\rangle\\
&+
\left\langle d_y(x),y_\delta(u)-y(x)\right\rangle .   
\end{aligned}
\end{equation*}
Then
\begin{equation}
\label{eq:g_delta_gradient_anchored_model}
g_\delta(x)=\nabla Q_{x,\delta}(x).
\end{equation}

Third,
\begin{equation}
\label{eq:g_delta_goldstein_anchored_model}
g_\delta(x)\in \partial_\delta Q_x(x),
\end{equation}
where $\partial_\delta Q_x(x)$ is the Goldstein $\delta$-subdifferential of
the scalar function $Q_x$. Moreover,
\begin{equation}
\label{eq:anchored_goldstein_set}
\begin{aligned}
\partial_\delta Q_x(x) &= d_x(x) \,+ \\
&\quad \conv \left( \bigcup_{\|s-x\|\le\delta} \left\{ V^\top d_y(x):V\in\clarke y(s) \right\} \right).
\end{aligned}
\end{equation}
\end{proposition}

\begin{proof} 
The complete proof is given in Appendix~\ref{app:proof_anchored_model}. Note that the Goldstein inclusion follows from the randomized smoothing--Goldstein relation applied to $Q_x$.
\end{proof}

Proposition~\ref{prop:anchored_model_interpretation} gives the precise
meaning of the Bi-ZOL approximate hypergradient. 
This geometric interpretation is shown in Fig.~\ref{fig:anchored_model_interpretation}.
The model $Q_x$ has the
same pointwise Clarke first-order geometry as the original hyperobjective at
the anchor point $x$, while its smoothed gradient is exactly $g_\delta(x)$.
Thus, $g_\delta(x)$ is not an arbitrary smoothed direction. It is the exact
gradient of a partially smoothed anchored chain-rule model and an element of
the Goldstein subdifferential of that same model.

This geometric interpretation will also be used in Section \ref{subsec:bizol_algorithm} to elaborate the algorithm.

\subsection{Structural Bias of a Bi-ZOL Direction}
\label{sec:structural_bias}

Proposition~\ref{prop:anchored_model_interpretation} shows that $g_\delta(x)$ belongs to the Goldstein subdifferential of the anchored model $Q_x$. 
To establish rigorous stationarity for the original problem, we now quantify how far this direction deviates from the pointwise Clarke chain-rule subdifferential of the true hyperobjective $\tilde\varphi(x)$.

For a fixed $x\in\mathcal X$, define the local set-variation modulus of the
Clarke generalized Jacobian by
\begin{definition}[Local set variation of the Clarke generalized Jacobian]
\label{def:local_set_variation}
For $x\in\mathcal X$ and $0<\delta\le\bar\delta$, define
\begin{equation}
\label{eq:clarke_jacobian_variation}
\omega_y^c(x,\delta)
\coloneqq
\sup_{\|s-x\|\le\delta}
\sup_{V\in\clarke y(s)}
\operatorname{dist}
\left(
V,\clarke y(x)
\right).
\end{equation}
\end{definition}

The quantity $\omega_y^c(x,\delta)$ measures how much the local generalized
Jacobian sets of the response map can move within a $\delta$-neighborhood of
$x$. It is a set-valued variation modulus. It is different from the ordinary
Lipschitz constant of $y(\cdot)$: the latter controls the size of elements in
$\clarke y(x)$, while $\omega_y^c(x,\delta)$ controls the local movement of
the set $\clarke y(\cdot)$ itself.

\begin{proposition}[Pointwise Clarke structural bias]
\label{prop:pointwise_clarke_bias}
Under Assumptions~\ref{ass:response} and~\ref{ass:upper_regular}, for every
$x\in\mathcal X$ and every $0<\delta\le\bar\delta$,
\begin{equation}
\label{eq:pointwise_clarke_bias}
\operatorname{dist}
\left(
g_\delta(x),
\partial^c\tilde\varphi(x)
\right)
\le
\|d_y(x)\|\,\omega_y^c(x,\delta).
\end{equation}
Consequently, for every fixed $x\in\mathcal X$,
\[
\operatorname{dist}
\left(
g_\delta(x),
\partial^c\tilde\varphi(x)
\right)
\to 0
\qquad
\text{as }\delta\downarrow 0.
\]
\end{proposition}

\begin{proof}
By Proposition~\ref{prop:anchored_model_interpretation},
$$g_\delta(x) \in d_x(x) + \conv \left( \bigcup_{\|s-x\|\le\delta} \left\{ V^\top d_y(x):V\in\clarke y(s) \right\} \right).$$
Hence $g_\delta(x)$ can be represented as a convex combination of elements of the form $d_x(x)+V^\top d_y(x)$ with $V\in\clarke y(s)$ and $\|s-x\|\le\delta$.

Since the Clarke generalized Jacobian $\clarke y(x)$ is a compact set, the distance $\operatorname{dist}(V,\clarke y(x))$ is attained. Thus, there exists $\bar V\in\clarke y(x)$ such that
$$\|V-\bar V\| = \operatorname{dist}(V,\clarke y(x)) \le \omega_y^c(x,\delta).$$
It follows that
$$
\begin{aligned}
&\operatorname{dist} \left( d_x(x)+V^\top d_y(x), \partial^c\tilde\varphi(x) \right) \\
&\le \left\| d_x(x)+V^\top d_y(x) - \left(d_x(x)+\bar V^\top d_y(x)\right) \right\| \\
&= \left\| (V-\bar V)^\top d_y(x) \right\| \\
&\le \|d_y(x)\|\omega_y^c(x,\delta).
\end{aligned}
$$

Since $\partial^c\tilde\varphi(x)$ is convex and $g_\delta(x)$ is a convex combination of such anchored chain-rule elements, the same upper bound holds for $g_\delta(x)$, yielding
$$\operatorname{dist}\left(g_\delta(x), \partial^c\tilde\varphi(x)\right) \le \|d_y(x)\|\,\omega_y^c(x,\delta).$$

Finally, the upper semicontinuity of the set-valued mapping $\clarke y(\cdot)$ implies
$$\omega_y^c(x,\delta)\to 0 \qquad \text{as }\delta\downarrow 0$$
for every fixed $x$. This completes the proof.
\end{proof}

If, in addition, the Clarke generalized Jacobian has a local Lipschitz-type
set variation at $x$, namely if there exist constants $L_{J,c}(x)>0$ and
$\rho_x>0$ such that
\begin{equation}
\label{eq:lipschitz_set_variation}
\sup_{V\in\clarke y(s)}
\operatorname{dist}
\left(
V,\clarke y(x)
\right)
\le
L_{J,c}(x)\|s-x\|,
\,
\forall s\in B_{\rho_x}(x),
\end{equation}
then, for every $0<\delta\le\rho_x$,
\begin{equation}
\label{eq:lipschitz_local_set_variation}
\omega_y^c(x,\delta)
\le
L_{J,c}(x)\delta.  
\end{equation}
Consequently, Proposition~\ref{prop:pointwise_clarke_bias} yields
\begin{equation}
\label{eq:pointwise_clarke_bias_linear_rate}
\operatorname{dist}
\left(
g_\delta(x),
\partial^c\tilde\varphi(x)
\right)
\le
\|d_y(x)\|L_{J,c}(x)\delta.
\end{equation}
Thus, the Bi-ZOL approximate hypergradient achieves a tight $O(\delta)$ pointwise structural bias whenever the local expansion of the Clarke generalized Jacobian set is linearly controlled.

The pointwise structural bias in Proposition~\ref{prop:pointwise_clarke_bias}
is controlled by the local movement of the generalized Jacobian set
$\clarke y(\cdot)$ and is measured by $\omega_y^c(x,\delta)$. 
The next results identify structured response mappings for which this
variation is controlled under a local $\delta$-regularity condition. 
We first give a piecewise $C^{1,1}$ bound, and then record the stronger zero-bias specialization for piecewise affine response mappings.

\begin{definition}[$\delta$-regularity for piecewise smooth responses] \label{def:delta_regular_piecewise_response} 
Suppose that, on a neighborhood of $\mathcal X$, the response map $y$ is continuous and piecewise smooth over a finite partition $\{\mathcal R_i\}_{i=1}^N$, namely \( y(u)=y_i(u),\, u\in \mathcal R_i, \) where each $y_i$ is smooth on a neighborhood of $\overline{\mathcal R_i}$. 
Define 
\[ I(x)\coloneqq\{i:x\in\overline{\mathcal R_i}\}, 
\qquad 
I_\delta(x)\coloneqq \{i:\mathcal R_i\cap B_\delta(x)\neq\emptyset\}. 
\] 
We say that $x$ is $\delta$-regular with respect to this partition if 
$I_\delta(x)\subseteq I(x).$ 
\end{definition}

Definition~\ref{def:delta_regular_piecewise_response} means the smoothing ball intersects only pieces that are already active at the anchor point.
The following proposition shows that $\delta$-regularity converts the piecewise \(C^{1,1}\) response into the bound needed in \eqref{eq:lipschitz_set_variation}.
\begin{proposition}[Piecewise $C^{1,1}$ response mappings] \label{prop:piecewise_c11_response_variation} 
Suppose that the response map $y$ is continuous and piecewise $C^{1,1}$ over a finite partition $\{\mathcal R_i\}_{i=1}^N$. 
Assume that, for each active piece, the Jacobian of $y_i$ satisfies 
\[ \|\jac y_i(s)-\jac y_i(t)\|\le L_i\|s-t\| \] 
whenever $s,t$ lie in a neighborhood of \(\overline{\mathcal R_i}\). 
If $x$ is $\delta$-regular, then 
\begin{equation} 
\label{eq:piecewise_c11_variation_bound} 
\omega_y^c(x,\delta) \le \left(\max_{i\in I(x)}L_i\right)\delta . \end{equation} 
Consequently, 
\begin{equation} 
\label{eq:piecewise_c11_clarke_bias} 
\operatorname{dist} \left( g_\delta(x),\partial^c\tilde\varphi(x) \right) 
\le
\|d_y(x)\| \left(\max_{i\in I(x)}L_i\right)\delta . 
\end{equation} 
\end{proposition}

\begin{proof} Fix $s\in B_\delta(x)$ and $V\in\clarke y(s)$. 
For a continuous piecewise \(C^{1,1}\) mapping, every element of $\clarke y(s)$ can be written as a convex combination of limiting Jacobians of pieces active at $s$. 
Hence there exist coefficients $\alpha_i\ge 0$, with $\sum_i\alpha_i=1$, such that 
\[ V=\sum_{i\in I(s)}\alpha_i \jac y_i(s). \] 
Since $x$ is $\delta$-regular and $s\in B_\delta(x)$, we have \(I(s)\subseteq I_\delta(x)\subseteq I(x)\). 
Therefore, 
\[ \bar V\coloneqq \sum_{i\in I(s)}\alpha_i \jac y_i(x) \] 
belongs to $\clarke y(x)$. 
It follows that 
\[ 
\begin{aligned} 
\operatorname{dist}(V,\clarke y(x)) &\le \|V-\bar V\| \\ 
&\le \sum_{i\in I(s)}\alpha_i \|\jac y_i(s)-\jac y_i(x)\| \\ 
&\le \left(\max_{i\in I(x)}L_i\right)\|s-x\|. 
\end{aligned} 
\] 
Taking the supremum over $s\in B_\delta(x)$ and $V\in\clarke y(s)$ gives \eqref{eq:piecewise_c11_variation_bound}. 
The bias bound \eqref{eq:piecewise_c11_clarke_bias} follows from Proposition~\ref{prop:pointwise_clarke_bias}. 
\end{proof}

A sharper picture is available when the response map is piecewise affine. In this case, the local Jacobians are constant on polyhedral cells, and the smoothing operation only
averages a finite set of affine sensitivities. This leads to exact pointwise
Clarke consistency in two important regimes.

\begin{corollary}[Piecewise affine (PWA) response mappings]
\label{cor:pwa_response_zero_bias}
Suppose that, on a neighborhood of $\mathcal X$, the response map $y$ is
continuous piecewise affine over a finite polyhedral partition
$\{\mathcal R_i\}_{i=1}^N$, namely
\[
y(u)=A_i u+b_i,\qquad u\in \mathcal R_i .
\]
Let $I(x)$ and $I_\delta(x)$ be defined as in
Definition~\ref{def:delta_regular_piecewise_response}. 
Then
\[
\begin{aligned}
\clarke y(x)&=\operatorname{conv}\{A_i:i\in I(x)\}, \\
\jac y_\delta(x)&\in
\operatorname{conv}\{A_i:i\in I_\delta(x)\}.    
\end{aligned}
\]
Consequently,
\[
g_\delta(x)
\in
d_x(x)+
\operatorname{conv}\{A_i^\top d_y(x):i\in I_\delta(x)\}.
\]
If $x$ is $\delta$-regular, i.e.,
\(
I_\delta(x)\subseteq I(x),
\)
then
\(
\omega_y^c(x,\delta)=0,
\,
g_\delta(x)\in\partial^c\tilde\varphi(x),
\)
which means the approximate hypergradient $g_\delta(x)$ exactly belongs to Clarke chain-rule subdifferential $\partial^c \tilde \varphi(x)$:
\[
\operatorname{dist}\left(g_\delta(x),\partial^c\tilde\varphi(x)\right)=0.
\]
Moreover, if $x$ lies in the interior of a single affine cell
$\mathcal R_{i_x}$, then for every
\(
0<\delta<
\operatorname{dist}\left(x,\mathbb R^n\setminus \mathcal R_{i_x}\right),
\)
we have
\[
g_\delta(x)
=
d_x(x)+A_{i_x}^\top d_y(x)
=
\nabla\tilde\varphi(x).
\]
\end{corollary}

\begin{proof}
The Clarke generalized Jacobian formula follows from the standard
representation of continuous piecewise affine mappings. 
Since the Jacobian is
constant on each polyhedral cell, the smoothed Jacobian $\jac y_\delta(x)$
is an average of the matrices $A_i$ over the cells intersected by
$B_\delta(x)$, which gives the stated inclusion.

If \(x\) is \(\delta\)-regular, then \(I_\delta(x)\subseteq I(x)\), and hence
\[
\jac y_\delta(x)\in
\operatorname{conv}\{A_i:i\in I(x)\}
=
\clarke y(x).
\]
Equivalently, \(\omega_y^c(x,\delta)=0\). The zero-bias claim then follows
from Proposition~\ref{prop:pointwise_clarke_bias}. The interior-cell claim is
the special case in which \(B_\delta(x)\) remains inside a single affine cell,
so \(\jac y_\delta(x)=A_{i_x}\) and the classical chain rule applies.
\end{proof}

A geometric interpretation of Corollary~\ref{cor:pwa_response_zero_bias} is plotted in Fig.~\ref{fig:PWA_two_panel}.
Corollary~\ref{cor:pwa_response_zero_bias} shows that the general
pointwise bias bound in Proposition~\ref{prop:pointwise_clarke_bias} can be
conservative for structured response maps.
This observation is particularly relevant because continuous PWA response
mappings arise naturally in many practical problems. 
A prominent
example is the solution map of strongly convex multi-parametric quadratic
programs (mpQPs) with affine constraints, which underlies explicit model
predictive control \cite{BemporadMorariDuaPistikopoulos2002,RawlingsMayneDiehl2017}. Similar
structures appear in power-system applications, where lower-level agents
solve constrained quadratic programs for demand response and energy
consumption scheduling \cite{Samadi2012}. In such settings, Bi-ZOL enjoys
a substantially stronger geometric interpretation than that suggested by
the general structural bias bound.

\begin{figure}[htbp]
    \centering
    \includegraphics[width=\columnwidth]{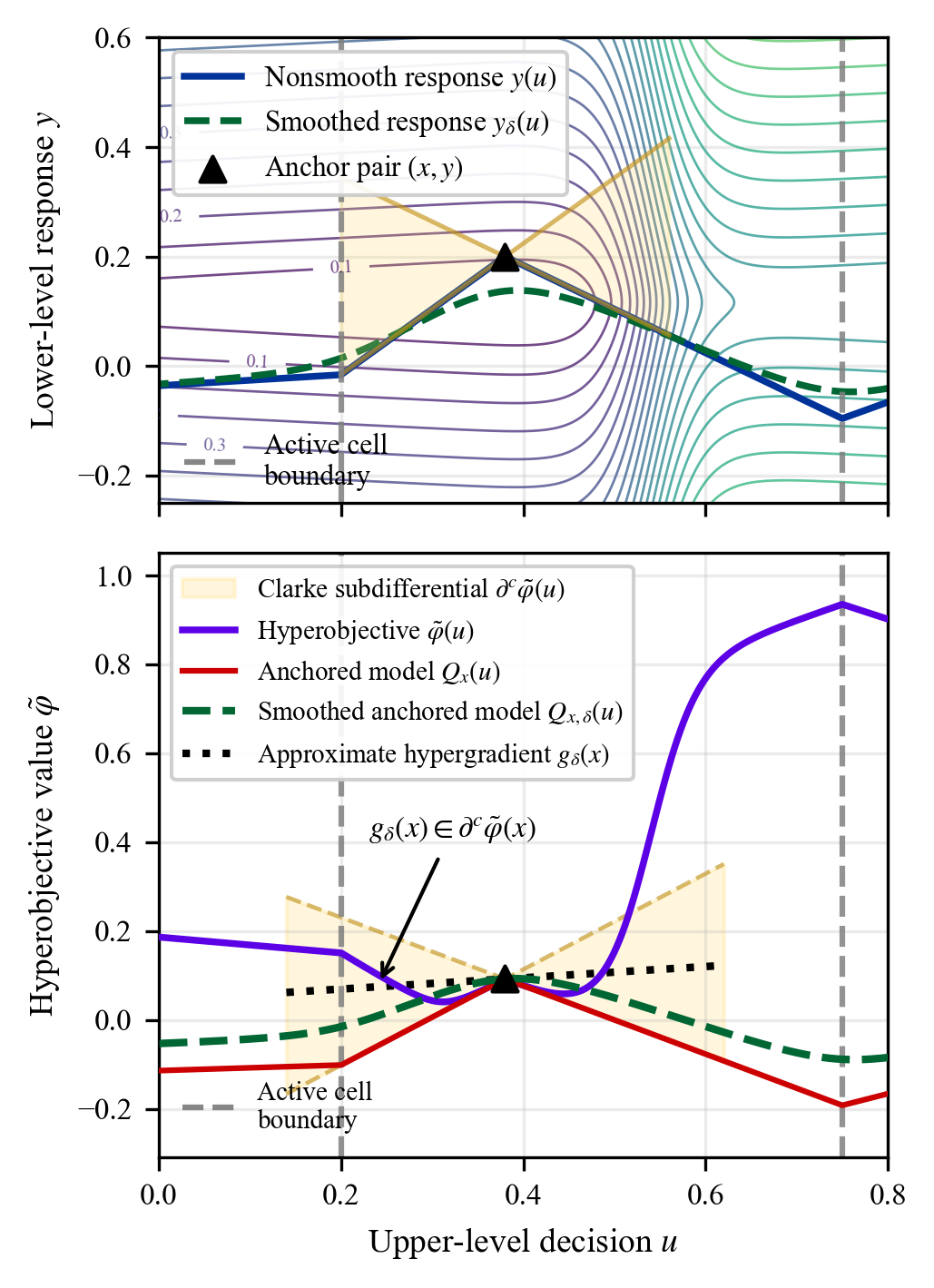}
    \caption{Visualization of the Bi-ZOL zero structural bias property at a non-differentiable switching point under a PWA response mapping. The top panel shows the PWA lower-level response $y(u)$ and its uniform smoothing $y_\delta(u)$ against the upper-level contours. 
    The bottom panel demonstrates that at the anchor point $x$, the gradient $g_\delta(x)$ of the smoothed anchored model $Q_{x,\delta}(u)$ falls exactly within the Clarke subdifferential $\partial^c\tilde\varphi(x)$ (yellow region) of the true hyperobjective. This visually verifies Corollary~\ref{cor:pwa_response_zero_bias}: $g_\delta(x) \in \partial^c\tilde\varphi(x)$ when the smoothing neighborhood $B_\delta(x)$ only intersects active cells.}
    \label{fig:PWA_two_panel}
\end{figure}

\subsection{Bi-ZOL Stationarity Certificates}
\label{sec:bizol_stationarity_certificates}

We now collect the stationarity certificates used to evaluate the
Bi-ZOL algorithm. 
Since the upper-level feasible set is constrained, we use
Frank--Wolfe gaps rather than unconstrained gradient norms.

When the reduced hyperobjective is differentiable, the standard stationarity
certificate for constrained smooth optimization is the Frank--Wolfe gap. We
recall this smooth certificate first, since it will serve as the reference case
for the Bi-ZOL certificate.

\begin{definition}[Smooth Frank--Wolfe stationarity certificate]
\label{def:smooth_fw_stationarity}
Suppose that the reduced hyperobjective \(\tilde\varphi(x)\) is differentiable at
\(x\in\mathcal X\). Define the smooth Frank--Wolfe gap by
\begin{equation}
\label{eq:smooth_fw_gap}
\mathcal G^{\rm sm}(x)
\coloneqq
\max_{z\in\mathcal X}
\left\langle z-x,-\nabla\tilde\varphi(x)\right\rangle .
\end{equation}
For \(\epsilon\ge 0\), we call \(x\in\mathcal X\) an
\(\epsilon\)-smooth Frank--Wolfe stationary point if
\begin{equation}
\label{eq:epsilon_smooth_fw_stationary}
\mathcal G^{\rm sm}(x)\le \epsilon .
\end{equation}
\end{definition}

For nonsmooth hyperobjectives, the gradient in the smooth
Frank--Wolfe gap is replaced by a Clarke subgradient. This gives the following
Clarke Frank--Wolfe certificate for the original bilevel reduced problem \cite{Liu2024zeroth}.

\begin{definition}[Clarke Frank--Wolfe stationarity certificate]
\label{def:clarke_fw_stationarity}
For \(x\in\mathcal X\), define the Clarke Frank--Wolfe gap of the original
reduced hyperobjective by
\begin{equation}
\label{eq:bl_fw_gap}
\mathcal G^c(x)
\coloneqq
\min_{g\in\BiC(x)}
\max_{z\in\mathcal X}
\left\langle z-x,-g\right\rangle .
\end{equation}
For \(\epsilon\ge 0\), we call \(x\in\mathcal X\) an
\(\epsilon\)-Clarke Frank--Wolfe stationary point if
\begin{equation}
\label{eq:epsilon_clarke_fw_stationary}
\mathcal G^c(x)\le \epsilon .
\end{equation}
\end{definition}

Based on these two state-of-art Frank--Wolfe stationarities, we define the certificate for our algorithm.
\begin{definition}[Bi-ZOL Frank--Wolfe stationarity certificate]
\label{def:bizol_fw_stationarity}
For \(x\in\mathcal X\) and \(0<\delta\le\bar\delta\), define the Bi-ZOL
Frank--Wolfe gap by
\begin{equation}
\label{eq:bizol_fw_gap}
\mathcal G_\delta(x)
\coloneqq
\max_{z\in\mathcal X}
\left\langle z-x,-g_\delta(x)\right\rangle .
\end{equation}
For \(\epsilon\ge 0\), we call \(x\in\mathcal X\) a
\((\delta,\epsilon)\)-Bi-ZOL stationary point if
\begin{equation}
\label{eq:delta_epsilon_bizol_stationary}
\mathcal G_\delta(x)\le \epsilon .
\end{equation}
\end{definition}

The Clarke certificate measures first-order stationarity of the original
reduced hyperobjective, while the Bi-ZOL certificate is the algorithmic
certificate induced by the approximate hypergradient \(g_\delta(x)\). The
following proposition connects these two certificates.

\begin{proposition}[Stationarity transfer]
\label{prop:stationarity_transfer}
Under Assumptions~\ref{ass:response}, \ref{ass:feasible}, and
\ref{ass:upper_regular}, for every \(x\in\mathcal X\) and every
\(0<\delta\le\bar\delta\),
\begin{equation}
\label{eq:stationarity_transfer_bias}
\mathcal G^c(x)
\le
\mathcal G_\delta(x)
+
D\|d_y(x)\|\,\omega_y^c(x,\delta),
\end{equation}
where $D=\operatorname{diam}(\mathcal X)$.

In particular, for a target accuracy \(\epsilon>0\), suppose that \(x\) is a
\((\delta,2\epsilon/3)\)-Bi-ZOL stationary point and that
\begin{equation}
\label{eq:stationarity_transfer_variation_condition}
D\|d_y(x)\|\,\omega_y^c(x,\delta)
\le
\frac{\epsilon}{3}.
\end{equation}
Then \(x\) is an \(\epsilon\)-Clarke Frank--Wolfe stationary point.

If, in addition, the local set variation satisfies
\(
\omega_y^c(x,\delta)\le L_{J,c}(x)\delta,
\)
then if $\|d_y(x)\|L_{J,c}(x)>0$, e.g., piecewise $C^{1,1}$ response with $\delta$-regularity, the condition in \eqref{eq:stationarity_transfer_variation_condition} is equivalent to 
\[ \delta \le \frac{\epsilon}{3D\|d_y(x)\|L_{J,c}(x)}. \] 
If $\|d_y(x)\|L_{J,c}(x)=0$, e.g., piecewise affine response with $\delta$-regularity, the structural-bias term vanishes and no extra upper bound on \(\delta\) is needed.
\end{proposition}

\begin{proof}
Let \(\bar g\in\partial^c\tilde\varphi(x)\) be a closest element to
\(g_\delta(x)\). By the definition of \(\mathcal G^c(x)\),
\[
\begin{aligned}
\mathcal G^c(x)
&\le
\max_{z\in\mathcal X}
\left\langle z-x,-\bar g\right\rangle \\
&\le
\max_{z\in\mathcal X}
\left\langle z-x,-g_\delta(x)\right\rangle
+
\max_{z\in\mathcal X}
\left\langle z-x,g_\delta(x)-\bar g\right\rangle \\
&\le
\mathcal G_\delta(x)
+
D\left\|g_\delta(x)-\bar g\right\|.
\end{aligned}
\]
Combining Proposition~\ref{prop:pointwise_clarke_bias} gives
\eqref{eq:stationarity_transfer_bias}.

If \(x\) is a \((\delta,2\epsilon/3)\)-Bi-ZOL stationary point, then
\(\mathcal G_\delta(x)\le 2\epsilon/3\). If
\eqref{eq:stationarity_transfer_variation_condition} also holds, then
\[
\mathcal G^c(x)
\le
\mathcal G_\delta(x)
+
D\|d_y(x)\|\,\omega_y^c(x,\delta)
\le
\frac{2\epsilon}{3}
+
\frac{\epsilon}{3}
=
\epsilon .
\]
Thus \(x\) is an \(\epsilon\)-Clarke Frank--Wolfe stationary point. Finally, if
\(\omega_y^c(x,\delta)\le L_{J,c}(x)\delta\), then subtituding $\omega_y^c(x,\delta)$ proves the last claim.
\end{proof}

Proposition~\ref{prop:stationarity_transfer} is a pointwise stationarity transfer. 
The convergence analysis of Bi-ZOL controls the Bi-ZOL Frank--Wolfe gap $\mathcal G_\delta$. 
The proposition does not claim that Bi-ZOL generally obtains Clarke Frank--Wolfe stationarity in finite time for arbitrary nonsmooth response maps. 
Instead, it states that when the output point also satisfies certain conditions of local generalized-Jacobian, the Bi-ZOL certificate induces a Clarke Frank--Wolfe certificate for the original reduced hyperobjective. 
This condition is satisfied, for example, at $\delta$-regular points of finite piecewise $C^{1,1}$ response mappings; piecewise affine response mappings form the zero-bias special case.

We next record the smooth-response interpretation of the Bi-ZOL certificate.
This result explains how the Bi-ZOL gap reduces to the standard smooth
Frank--Wolfe gap when the lower-level response is smooth.

\begin{proposition}[Smooth-response reduction of Bi-ZOL stationarity]
\label{prop:smooth_response_bizol_stationarity}
Suppose that the response map \(y(\cdot)\) is smooth and that \(\jac y(\cdot)\) is
\(L_J\)-Lipschitz on the query region.  If \(x\) is a
\((\delta,2\epsilon/3)\)-Bi-ZOL stationary point with
\[
0<\delta\le \frac{\epsilon}{3DM_yL_J},
\]
then \(x\) is an \(\epsilon\)-stationary point in the standard smooth
Frank--Wolfe sense.
\end{proposition}
\begin{proof}
Since \(y(\cdot)\) is smooth, the hyperobjective is differentiable. 
Recall
\[
g_\delta(x)
=
d_x(x)+\jac y_\delta(x)^\top d_y(x),
\,
\jac y_\delta(x)
=
\mathbb E_\xi[\jac y(x+\delta\xi)] .
\]
Thus, by the \(L_J\)-Lipschitz continuity of \(\jac y\),
\begin{equation*}
\begin{aligned}
\|g_\delta(x)-\nabla\tilde\varphi(x)\|
&\le
\|d_y(x)\|
\mathbb E_\xi
\left[
\|\jac y(x+\delta\xi)-\jac y(x)\|
\right]  \\
&\le
M_yL_J\delta .
\end{aligned}
\end{equation*}
By Definition~\ref{def:smooth_fw_stationarity},
\begin{equation*}
\begin{aligned}
\mathcal G^{\rm sm}(x)
&\le
\max_{z\in\mathcal X}
\left\langle z-x,-g_\delta(x)\right\rangle
+
D\|g_\delta(x)-\nabla\tilde\varphi(x)\| \\
&\le
\mathcal G_\delta(x)+DM_yL_J\delta .
\end{aligned}
\end{equation*}
If \(x\) is a \((\delta,2\epsilon/3)\)-Bi-ZOL stationary point and
\(\delta\le \epsilon/(3DM_yL_J)\), then
\[
\max_{z\in\mathcal X}
\left\langle z-x,-\nabla\tilde\varphi(x)\right\rangle
\le
\frac{2\epsilon}{3}+\frac{\epsilon}{3}
=
\epsilon.
\]
\end{proof}

\section{Solution Methods}
\label{sec:solution_methods}

The previous section defines the approximate hypergradient
\[
g_\delta(x)
=
\nabla_1\varphi(x,y(x))
+
\jac y_\delta(x)^\top \nabla_2\varphi(x,y(x)),
\]
which preserves the current upper-level first-order information and surrogates
only the response sensitivity. This section shows how to estimate
\(g_\delta(x)\) using zeroth-order response queries and how the resulting
Frank--Wolfe scheme controls the approximate stationarity measure
\(\mathcal G_\delta(x)\) defined in \eqref{eq:bizol_fw_gap}.

\subsection{Zeroth-order approximation of the smoothed response Jacobian}
\label{subsec:zo_jacobian_approximation}

The direction \(g_\delta(x)\) is not directly computable under the response
oracle model, because the leader does not have access to the smoothed
Jacobian \(\jac y_\delta(x)\). Bi-ZOL estimates this Jacobian using two-point
random perturbations of the lower-level response.

For any \(x\in\mathcal X\) and
\(w\sim{\rm Unif}(\mathbb S^{n-1})\), define
\begin{equation}
\label{eq:single_sample_jac_estimator}
\apjac y_\delta(x;w)
\coloneqq
\frac{n}{2\delta}
\left(
y(x+\delta w)-y(x-\delta w)
\right)w^\top
\in\mathbb R^{m\times n}.
\end{equation}
The corresponding single-sample Bi-ZOL hypergradient estimator is
\begin{equation}
\label{eq:single_sample_hg_estimator}
\widehat g_\delta(x;w)
\coloneqq
\nabla_1\varphi(x,y(x))
+
\apjac y_\delta(x;w)^\top
\nabla_2\varphi(x,y(x)).
\end{equation}
For a mini-batch size \(B\ge 1\), let
\(\{w_b\}_{b=1}^B\) be independent samples from
\({\rm Unif}(\mathbb S^{n-1})\). Define
\begin{equation}
\label{eq:batch_jac_estimator}
\apjac y_{\delta,B}(x)
\coloneqq
\frac{1}{B}
\sum_{b=1}^B
\apjac y_\delta(x;w_b),
\end{equation}
and
\begin{equation}
\label{eq:batch_hg_estimator}
\widehat g_{\delta,B}(x)
\coloneqq
\nabla_1\varphi(x,y(x))
+
\apjac y_{\delta,B}(x)^\top
\nabla_2\varphi(x,y(x)).
\end{equation}
The case \(B=1\) gives the single-sample estimator.

The estimator in \eqref{eq:batch_hg_estimator} only randomizes the response sensitivity term. 
The following lemma shows that this estimator is unbiased for the Bi-ZOL approximate hypergradient \(g_\delta(x)\) and gives the variance bound used in the convergence analysis.
\begin{lemma}[Bi-ZOL hypergradient estimator]
\label{lem:solution_hg_estimator}
Under Assumptions~\ref{ass:response} and~\ref{ass:upper_regular}, for every
\(x\in\mathcal X\),
\begin{equation}
\label{eq:hg_estimator_unbiased}
\mathbb E
\left[
\widehat g_{\delta,B}(x)
\right]
=
g_\delta(x).
\end{equation}
Moreover,
\begin{equation}
\label{eq:hg_estimator_variance}
\mathbb E
\left[
\left\|
\widehat g_{\delta,B}(x)-g_\delta(x)
\right\|^2
\right]
\le
\frac{\sigma_\delta^2}{B},
\end{equation}
where
\begin{equation}
\label{eq:sigma_delta_def}
\sigma_\delta^2
\coloneqq
c_{\rm var}nM_y^2L_y^2,
\qquad
c_{\rm var}=16\sqrt{2\pi}.
\end{equation}
\end{lemma}

\begin{proof}
The proof is given in Appendix~\ref{app:proof_bizol_hypergradient_estimator}.
\end{proof}

\subsection{Bi-ZOL algorithm}
\label{subsec:bizol_algorithm}

Bi-ZOL alternates between observing the lower-level response, estimating the
smoothed response Jacobian, and performing a Frank--Wolfe update over
\(\mathcal X\). The algorithm is stated in Algorithm~\ref{alg:bizol} in a mini-batch form.

Recall the geometric interpretation in Section \ref{sec:anchored_chain_rule_model}, Bi-ZOL actually performs a moving anchored linearization: at each
iterate, it constructs a local first-order model $Q_x$ that is anchored at the
the latest observed response pair (Line 3-4), and then use a two-point estimator to approximate the response sensitivity (Line 5-7).
The estimated response sensitivity is then used to form the approximate hypergradient estimator (Line 8).
Hence the resulting direction remains tied to the current pair $(x,y(x))$ and to the bilevel
chain-rule geometry described by Proposition~\ref{prop:bl_chain_rule}.
Then Bi-ZOL performs a Frank--Wolfe update over $\mathcal X$ to find the next iterate (Line 9-10).
With the new iterate, Bi-ZOL constructs a new anchored linearization and repeats the process.

\begin{algorithm}[htb]
\caption{Bilevel Zeroth-Order Learning (Bi-ZOL)}
\label{alg:bizol}
\begin{algorithmic}[1]
\STATE \textbf{Input:} initial point \(x_0\in\mathcal X\), smoothing radius
\(\delta\in(0,\bar\delta]\), stepsizes \(\gamma\in(0,1]\),
mini-batch size \(B\ge1\).
\FOR{\(k=0,1,2,\ldots\)}
    \STATE Query the lower-level response \(y_k\leftarrow y(x_k)\).
    \STATE Compute
    \[
    d_x(x_k)\leftarrow \nabla_1\varphi(x_k,y_k),
    \,
    d_y(x_k)\leftarrow \nabla_2\varphi(x_k,y_k).
    \]
    \STATE Sample \(w_{k,b}\sim{\rm Unif}(\mathbb S^{n-1})\),
    \(b=1,\ldots,B\), independently.
    \STATE Query perturbed responses
    \[
    y_{k,b}^{+}\leftarrow y(x_k+\delta w_{k,b}),
    \,
    y_{k,b}^{-}\leftarrow y(x_k-\delta w_{k,b}).
    \]
    \STATE Form the zeroth-order Jacobian estimator
    \[
    \apjac y_{\delta,B}(x_k)
    \leftarrow
    \frac{1}{B}
    \sum_{b=1}^B
    \frac{n}{2\delta}
    \left(
    y_{k,b}^{+}-y_{k,b}^{-}
    \right)w_{k,b}^\top .
    \]
    \STATE Form the approximate hypergradient estimator
    \[
    \widehat g_k
    \leftarrow
    d_x(x_k)+\apjac y_{\delta,B}(x_k)^\top d_y(x_k) .
    \]
    \STATE Compute the Frank--Wolfe oracle
    \[
    z_k
    \in
    \arg\max_{z\in\mathcal X}
    \left\langle z-x_k,-\widehat g_k\right\rangle .
    \]
    \STATE Update
    \[
    x_{k+1}
    \leftarrow
    x_k+\gamma(z_k-x_k).
    \]
\ENDFOR
\end{algorithmic}
\end{algorithm}

Note that the response queries \(x_k\pm\delta w_{k,b}\) are used only to estimate the
local response sensitivity. They are not new upper-level iterates. The
Frank--Wolfe update keeps \(x_{k+1}\in\mathcal X\) because
\(\mathcal X\) is convex and \(\gamma\in(0,1]\).

\subsection{Convergence analysis and oracle complexity}
\label{subsec:convergence_complexity}

Recall from Section \ref{sec:bizol_stationarity_certificates} that the convergence certificate is to find a \((\delta,\epsilon)\)-Bi-ZOL stationary point.
To prove the convergence, we first define a smooth ancillary function
\begin{equation}
\label{eq:solution_ancillary_function}
\bar\varphi_\delta(x)
\coloneqq
\varphi(x,y_\delta(x)).
\end{equation}
The function $\bar\varphi_\delta(x)$ is essentially another smooth surrogate of $\tilde \varphi(x)$ by substituting the nonsmooth $y(x)$ with $y_\delta(x)$.
But we do not aim to optimize this smooth surrogate, and it is introduced strictly as an analysis device, allowing us to utilize its smoothness constant for the convergence proof. 
The algorithmic direction
is \(g_\delta(x)\), not \(\nabla\bar\varphi_\delta(x)\). The difference
between these two quantities is a bias term induced by evaluating the
upper-level partial derivatives at \(y(x)\) in \(g_\delta(x)\) and at
\(y_\delta(x)\) in \(\nabla\bar\varphi_\delta(x)\).

Define
\[
M_J
\coloneqq
\sup_{x\in\mathcal X}
\|\jac y_\delta(x)\|.
\]
Since \(y_\delta\) is \(L_y\)-Lipschitz, \(M_J\le L_y\). Also define
\begin{equation}
\label{eq:Ay_def}
A_y
\coloneqq
L_{1x}+L_{1y}L_y
+
M_J(L_{2x}+L_{2y}L_y).
\end{equation}
The following lemma gives the smoothness constant used in the descent
argument.

\begin{lemma}[Smoothness of the ancillary function]
\label{lem:solution_ancillary_smoothness}
Under Assumptions~\ref{ass:response}, \ref{ass:feasible}, and
\ref{ass:upper_regular}, the function
\(\bar\varphi_\delta\) is differentiable on \(\mathcal X\), and
\(\nabla\bar\varphi_\delta\) is Lipschitz continuous with constant
\begin{equation}
\label{eq:Lbar_delta_def}
L_{\bar\varphi,\delta}
\coloneqq
A_y
+
\frac{cM_yL_y\sqrt{n}}{\delta},
\end{equation}
\end{lemma}

\begin{proof}
Please see Appendix~\ref{app:proof_solution_ancillary_smoothness}.
\end{proof}

In the following Lemma, we compute the bias between \(\nabla\bar\varphi_\delta\) and \(g_\delta\), which appears in the convergence analysis.
\begin{lemma}[Bias between \(\nabla\bar\varphi_\delta\) and \(g_\delta\)]
\label{lem:solution_bias}
Under Assumptions~\ref{ass:response} and~\ref{ass:upper_regular}, for every
\(x\in\mathcal X\),
\begin{equation}
\label{eq:solution_bias_bound}
\left\|
\nabla\bar\varphi_\delta(x)-g_\delta(x)
\right\|
\le
C_{\rm bias}\delta,
\end{equation}
where
\begin{equation}
\label{eq:Cbias_def}
C_{\rm bias}
\coloneqq
(L_{1y}+M_JL_{2y})L_y .
\end{equation}
Equivalently, defining
\begin{equation}
\label{eq:kappa_bias_def}
\kappa_{\rm bias}
\coloneqq
DC_{\rm bias},
\end{equation}
the bias contribution in the Frank--Wolfe gap bound is at most
\(\kappa_{\rm bias}\delta\).
\end{lemma}

\begin{proof}
The complete proof is given in Appendix~\ref{app:proof_solution_bias}.
\end{proof}

\begin{theorem}[Bi-ZOL convergence]
\label{thm:solution_bizol_convergence}
Suppose Assumptions~\ref{ass:response}, \ref{ass:feasible}, and
\ref{ass:upper_regular} hold. 
Fix a smoothing radius \(0<\delta\le\bar\delta\), a batch size
\(B\in\mathbb N\), a horizon \(T\in\mathbb N\), and a constant stepsize
\(\gamma\in(0,1]\). Let \(\{x_k\}_{k=0}^{T}\) be generated by
Algorithm~\ref{alg:bizol} with these parameters.
Then
\begin{equation}
\label{eq:solution_main_convergence_bound}
\frac{1}{T}
\sum_{k=0}^{T-1}
\mathbb E
\left[
\mathcal G_\delta(x_k)
\right]
\le
\frac{\Delta_\delta}{\gamma T}
+
\frac{L_{\bar\varphi,\delta}D^2}{2}\gamma
+
\kappa_{\rm bias}\delta
+
\frac{D\sigma_\delta}{\sqrt B},
\end{equation}
where
\[
\Delta_\delta
\coloneqq
\bar\varphi_\delta(x_0)
-
\min_{x\in\mathcal X}
\bar\varphi_\delta(x),
\]
\(L_{\bar\varphi,\delta}\) is defined in
\eqref{eq:Lbar_delta_def}, \(\kappa_{\rm bias}\) is defined in
\eqref{eq:kappa_bias_def}, and \(\sigma_\delta\) is defined in
\eqref{eq:sigma_delta_def}.
\end{theorem}

\begin{proof}
[Proof sketch] The proof uses \(\bar\varphi_\delta\) as a smooth descent function. By Lemma~\ref{lem:solution_ancillary_smoothness}, one step of the Frank--Wolfe update gives a descent inequality with curvature term \(L_{\bar\varphi,\delta}D^2\gamma^2/2\). Since the update direction is computed from the stochastic estimator \(\widehat g_{\delta,B}(x_k)\) rather than from \(\nabla\bar\varphi_\delta(x_k)\), the descent bound contains two additional errors: the bias \(\|\nabla\bar\varphi_\delta(x_k)-g_\delta(x_k)\|\), controlled by Lemma~\ref{lem:solution_bias}, and the estimation error \(\|\widehat g_{\delta,B}(x_k)-g_\delta(x_k)\|\), controlled by Lemma~\ref{lem:solution_hg_estimator}. Summing the resulting inequality over \(k=0,\ldots,T-1\) gives \eqref{eq:solution_main_convergence_bound}. The complete proof is given in Appendix~\ref{app:proof_solution_bizol_convergence}. \end{proof}

\begin{remark}
    \label{rem:solution_bizol_convergence}
    In Theorem~\ref{thm:solution_bizol_convergence}, we use the average of the expected Bi-ZOL Frank--Wolfe gap as the convergence measure. The interpretation is that it equals $\mathbb{E}[\mathcal{G}_\delta(x_R)]$, where $x_R$ is selected uniformly at random from the candidate solutions $\{x_k\}_{k=0}^{T-1}$. Furthermore, the right-hand side of \eqref{eq:solution_main_convergence_bound} is also the upper bound on $\mathbb{E}[\min_{k=0,\dots,T-1} \mathcal{G}_\delta(x_k)]$, i.e., the expected value of the minimum Bi-ZOL Frank--Wolfe gap. 
\end{remark}
The bounds above are stated for the Bi-ZOL certificate
\(\mathcal G_\delta\). The induced Clarke-type certificate follows from
Proposition~\ref{prop:stationarity_transfer} and contains the additional
structural smoothing term controlled by
\(\omega_y^c(x,\delta)\).

We now translate Theorem~\ref{thm:solution_bizol_convergence} into a
response-oracle complexity bound. 
\begin{corollary}[Response-oracle complexity of Bi-ZOL]
\label{cor:bizol_response_oracle_complexity}
Suppose the conditions of Theorem~\ref{thm:solution_bizol_convergence} hold. Let $\epsilon > 0$ be a target Bi-ZOL Frank--Wolfe accuracy, and let $x_R$ be sampled uniformly at random from $\{x_0,\ldots,x_{T-1}\}$.
By choosing the smoothing radius 
\(
\delta
=
\Theta\left(
\frac{\epsilon}{\kappa_{\rm bias}}
\right),
\) 
the batch size 
\(
B
=
\Theta\left(
\frac{D^2\sigma_\delta^2}{\epsilon^2}
\right)
=
\Theta\left(
\frac{D^2nM_y^2L_y^2}{\epsilon^2}
\right),
\)
and the stepsize 
\(
\gamma
=
\min\left\{
1,\,
\sqrt{
\frac{2\Delta_\delta}
{L_{\bar\varphi,\delta}D^2T}
}
\right\},
\)
the algorithm achieves $\mathbb{E}[\mathcal{G}_\delta(x_R)] \le \mathcal{O}(\epsilon)$ in 
\[
T
=
\Theta\left(
\Delta_\delta D^2A_y\,\epsilon^{-2}
+
\Delta_\delta D^2M_yL_y\sqrt n\,\kappa_{\rm bias}\,\epsilon^{-3}
\right)
\] iterations. 

Consequently, the leading response-oracle complexity is bounded by
\begin{equation}
\label{eq:cor_leading_response_oracle_complexity_bizol}
N_y^{\rm BiZOL} = \mathcal{O}\left( n^{3/2} M_y^3 L_y^3 \kappa_{\rm bias} \epsilon^{-5} \right),
\end{equation}
where constants independent of $n$, $m$, and $\epsilon$ are suppressed.
\end{corollary}
\begin{proof}
By Theorem~\ref{thm:solution_bizol_convergence} and the uniform random
choice of \(R\),
\[
\mathbb E[\mathcal G_\delta(x_R)]
\le
\frac{\Delta_\delta}{\gamma T}
+
\frac{L_{\bar\varphi,\delta}D^2}{2}\gamma
+
\kappa_{\rm bias}\delta
+
\frac{D\sigma_\delta}{\sqrt B}.
\]
The choice of $\delta$ makes
\(\kappa_{\rm bias}\delta=O(\epsilon)\). 
The
choice of $B$ makes
\(D\sigma_\delta/\sqrt B=O(\epsilon)\).
With the choice of $\gamma$, the two
optimization terms are bounded by
\(
O\left(
D\sqrt{\frac{\Delta_\delta L_{\bar\varphi,\delta}}{T}}
\right).
\)
Since
\(
L_{\bar\varphi,\delta}
=
A_y
+
\frac{cM_yL_y\sqrt n}{\delta},
\)
the choice of $\delta$ gives
\(
L_{\bar\varphi,\delta}
=
A_y
+
O\left(
\frac{M_yL_y\sqrt n\,\kappa_{\rm bias}}{\epsilon}
\right).
\)
Therefore, the stated choice of \(T\) makes
\(
D\sqrt{\frac{\Delta_\delta L_{\bar\varphi,\delta}}{T}}
=
O(\epsilon).
\)
Combining this with the choices of \(\delta\) and \(B\) gives
\(
\mathbb E[\mathcal G_\delta(x_R)]\le O(\epsilon).
\)

    Since each iteration uses \(2B\) perturbed response queries and one response
query at the current point \(x_k\), the total number of response-oracle calls
satisfies
\(
N_y^{\rm BiZOL}
=O(TB)
=
O\left(
\Delta_\delta D^4 nM_y^2L_y^2A_y\,\epsilon^{-4}
+
\Delta_\delta D^4 n^{3/2}M_y^3L_y^3
\kappa_{\rm bias}\,\epsilon^{-5}
\right).
\)
In particular, for small \(\epsilon\), the smoothing-induced term gives the
leading response-oracle complexity.
\end{proof}

If
\(
M_y=O(1),
\,
L_y=O(1),
\,
\kappa_{\rm bias}=O(1),
\)
then the leading response-oracle complexity has no explicit dependence on
the response dimension \(m\), and
\[
N_y^{\rm BiZOL}
=
O\left(
n^{3/2}\epsilon^{-5}
\right).
\]
If the response coupling scales as
\(
L_y=\Theta(\rho_y),
\,
M_J=\Theta(\rho_y),
\)
then
\(
\kappa_{\rm bias}
=
D(L_{1y}+M_JL_{2y})L_y
=
\Theta(\rho_y^2)
\)
when \(D,L_{1y},L_{2y}=O(1)\). With \(M_y=O(1)\), the leading response-oracle complexity becomes
\[
N_y^{\rm BiZOL}
=
O\left(
n^{3/2}\rho_y^5\epsilon^{-5}
\right).
\]
In particular, if \(\rho_y=\sqrt m\), then
\[
N_y^{\rm BiZOL}
=
O\left(
n^{3/2}m^{5/2}\epsilon^{-5}
\right).
\]
These dimension scalings arise when the scalarized response sensitivity
\(M_yL_y\) or the bias scale \(\kappa_{\rm bias}\) grows with the response
dimension \(m\).

\begin{remark}[Response-size dependence of the response sensitivity]
\label{rem:response_size_dependence_of_response_sensitivity}
The constant \(L_y\) should not always be interpreted as a
dimension-free numerical constant. In a lower-level setting where each element in the response vector $y$ is an individual response,
\(y(x)=\mathrm{col}(y_1(x),\ldots,y_m(x))\), the Lipschitz constant of the stacked response may scale with the size of $m$.
For example, if each local response satisfies
\[
\|y_i(x)-y_i(x')\|\le \ell \|x-x'\|,
\qquad i=1,\ldots,m,
\]
then
\[
\|y(x)-y(x')\|
\le
\sqrt{m}\ell\|x-x'\|.
\]

This dependence can be mitigated when the response has normalized
average-interaction or mean-field scaling, or when the upper-level
objective is an average objective. For instance, if
\[
\varphi(x,y)=\frac1m\sum_{i=1}^m \varphi_i(x,y_i)
\]
and the local gradients \(\nabla_{y_i}\varphi_i\) are uniformly bounded,
then \(M_y=\|\nabla_2\varphi(x,y)\|=O(m^{-1/2})\). In this case, the
product \(M_yL_y\) may remain \(O(1)\) even when \(L_y=O(\sqrt m)\).
Therefore, the relevant scalability factor in Bi-ZOL is not only
\(L_y\), but the combined quantities \(M_yL_y\), and \(\kappa_{\rm bias}\).
\end{remark}

\section{Simulation Results}
\label{sec:simulation_results}
In this section, we numerically investigate the performance of Algorithm~\ref{alg:bizol} by deploying it on a general incentive-based tracking problem with capacity-limited responses. 
First, we test the convergence of Bi-ZOL by showing the Bi-ZOL Frank--Wolfe gap decreases along the generated trajectory.
Second, we show the decrease of the hyperobjective value with different batch sizes and compare the results with those of VZO.
Third, we illustrate the scalability of the solution algorithm by showing tests on different problem dimensions.

\subsection{Simulation setup}
We consider a bilevel tracking problem in which a system operator chooses an upper-level control signal $x\in\mathcal X$, and a lower-level subsystem produces a constrained response $y(x)$. 
The upper-level signal can be interpreted as an incentive, and the lower-level response is obtained from a constrained quadratic discomfort minimization with a linear incentive induced by $x$.

\subsubsection{Upper level}
The variable $x$ represents an upper-level signal and it is limited by $x\in\mathcal{X}$. 
The upper-level objective penalizes control effort, tracking error of the induced response, and control-response coupling (payment from upper level to lower level).
The objective function is expressed as $\varphi \left( x,y \right) =\frac{\mu _x}{2}\left\| x-x_{\rm ref} \right\| _{2}^{2}+\frac{\kappa}{2}\left\| Hy-y_{\rm ref} \right\| _{2}^{2}+\beta x^{\top}Cy$, where $x_{\rm ref}$ is a nominal upper-level signal, $y_{\rm ref}$ is the desired response reference, $H$ maps the response to the tracked output, and $C$ describes the coupling between the upper-level signal and the lower-level response.
The feasible set of the upper-level variable is a box $\mathcal X=[-R_x,R_x]^n$.

\subsubsection{Lower level}
The lower-level model represents a constrained response of flexible agents. 
Each agent has a quadratic discomfort cost for deviating from its baseline response $d$ and receives a linear incentive induced by the upper signal.
Thus the lower level is to minimize $g(x,y) =\frac{1}{2}\left\| y-d \right\| _{2}^{2}-\left( Ax \right) ^{\top}y$, subject to the capacity limit $y\in\mathcal{Y}$.
Given $x$, completing the square gives the equivalent projection response map, $y(x) =\Pi _{\mathcal{Y}}( d+Ax )$.
Therefore, the nonsmoothness of the response map is generated by the physical capacity set $\mathcal Y$.

We consider two response sets, both of which yield nonsmooth response maps: 
\begin{itemize}
    \item A heterogeneous box constraint $\mathcal{Y} _{box}=\prod_{i=1}^m{\left[ -\bar{y}_i,\bar{y}_i \right]}$, which represents individual capacity limits and gives a PWA response map;
    \item An Euclidean ball constraint $\mathcal{Y} _{ball}=\left\{ y\in \mathbb{R} ^m:\left\| y \right\| _2\le r \right\}$, which represents an aggregate response budget.
    The corresponding response map has a curved active boundary and is not PWA.
\end{itemize}

In the following simulations, we mainly compare Bi-ZOL with VZO.
VZO applies a zeroth-order estimator to the whole hyperobjective $\tilde \varphi (x)$.
In contrast, Bi-ZOL uses the exact upper-level derivatives at the current response and applies zeroth-order sampling only to the response-sensitivity term.
For each comparison, the two methods use the same smoothing radius $\delta$, the same constant Frank--Wolfe stepsize $\gamma$, and the same batch size $B$ in the two-point estimator.

Each test curve is obtained from 10 independent runs, which use different random samplings. 
For a fair comparison, the same problem instance and the same initial point are used for Bi-ZOL and VZO within each run, while the random zeroth-order directions are regenerated according to the corresponding sampling seed. 
The solid curves report the empirical mean, and the shaded regions report the standard deviation across runs.

\subsection{Convergence behavior}
We first consider the box-constrained lower-level problem, whose response map is PWA.
In practice, this setting corresponds to individual capacity limits of the lower-level agents.
In this experiment, we use dimensions $n=20$ and $m=200$.
The algorithm stepsize and smoothing radius is set as $\gamma=0.1$ and $\delta=0.08$.
To monitor whether the trajectory interacts with nonsmooth regions, we also record active-set related diagnostics, such as the fraction of saturated response and the frequency with which the sampled perturbations cross active-set boundaries. 
These diagnostics confirm that the response constraints are active along the trajectory and that the experiment is not reduced to a smooth unconstrained response case.

We illustrate the convergence behavior of Bi-ZOL in Fig.~\ref{fig:PWA_stationarity} using best-so-far validation Bi-ZOL Frank--Wolfe gap.
At each $x_k$, we compute the validation hypergradient estimator, which is 
\(
\widehat g_{\delta}^{\operatorname{val}}(x_k) = d_x(x_k) + \frac{1}{B_{\operatorname{val}}} \sum_{j=1}^{B_{\operatorname{val}}} \frac{n}{2\delta} \left\langle d_y(x_k), y(x_k+\delta w_j)-y(x_k-\delta w_j) \right\rangle w_j,
\)
to approximate the exact value of $g_\delta(x_k)$. 
We set $B_{\operatorname{val}}=200$.
The corresponding validation Frank--Wolfe gap is $\widehat{\mathcal G}_{\delta}^{\operatorname{val}}(x_k) = \max_{s\in\mathcal X} \left\langle x_k-s, \widehat g_{\delta}^{\operatorname{val}}(x_k) \right\rangle$.
To reduce the effect of randomness, we plot the best-so-far Frank--Wolfe gap $\widehat{\mathcal G}_{\delta,\operatorname{best}}^{\operatorname{val}}(x_k) = \min_{0\le t\le k} \widehat{\mathcal G}_{\delta}^{\operatorname{val}}(x_t)$ to depict the convergence behavior of $\mathcal{G}_\delta(x)$.
To validate our stationarity transfer in Proposition~\ref{prop:stationarity_transfer}, we also compute the Clarke Frank--Wolfe gap along the same trajectories and plot its best-so-far trajectory $\mathcal G^c_{\operatorname{best}}(x_k)
=
\min_{0\le t\le k}\mathcal G^c(x_t)$.

\begin{figure}[htbp]
    \centering
    \includegraphics[width=\columnwidth]{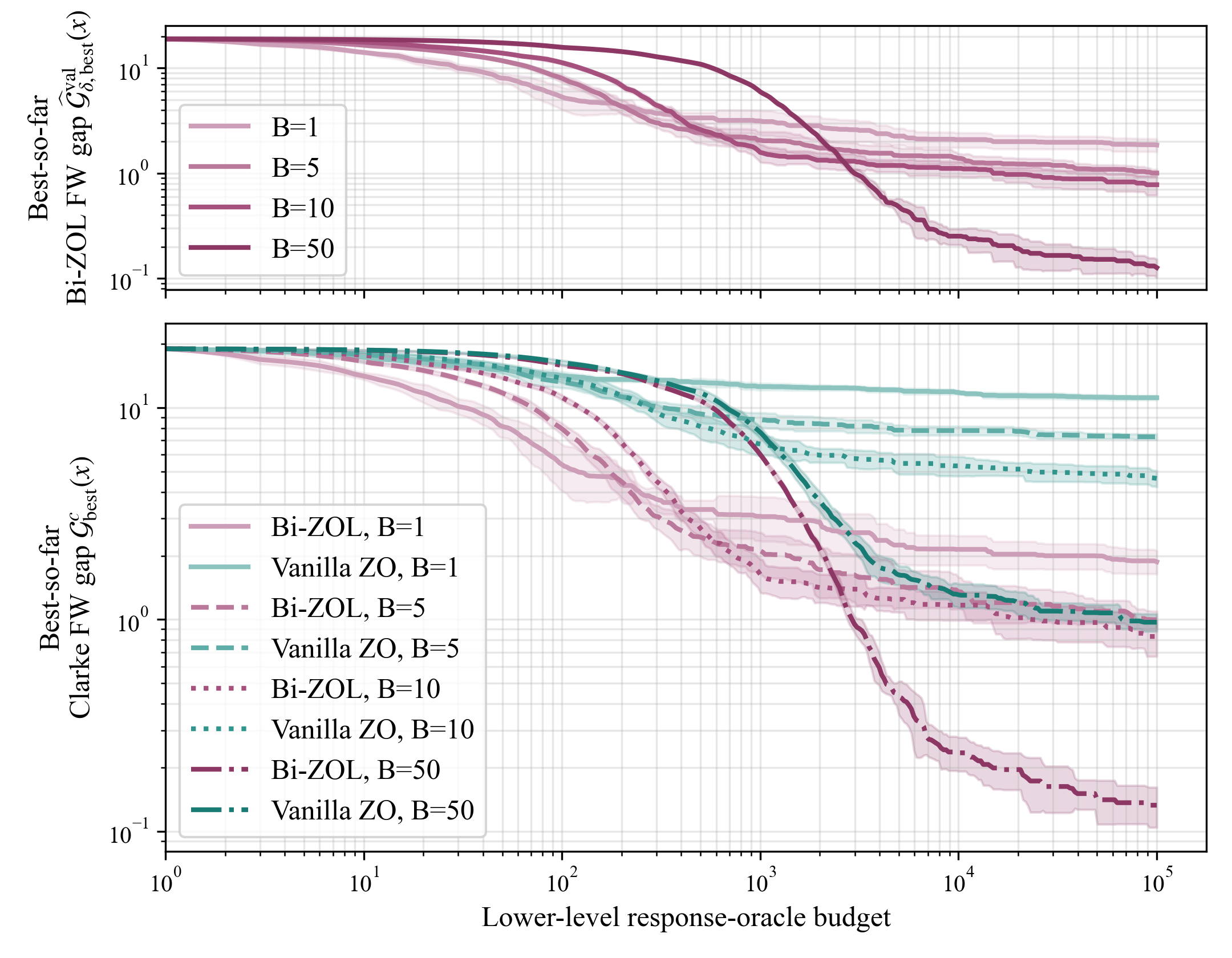}
    \caption{Stationarity gap curves on the PWA response experiment under different response-oracle batch sizes $B$.
    The upper panel reports the Bi-ZOL best-so-far validation Frank--Wolfe gap
    $\widehat{\mathcal G}_{\delta,\operatorname{best}}^{\operatorname{val}}(x_k)$,
    computed with $B_{\operatorname{val}}=200$ validation directions.
    The lower panel reports the exact best-so-far Clarke Frank--Wolfe gap
    $\mathcal G^c_{\operatorname{best}}(x_k)$ for both Bi-ZOL and VZO.
    The horizontal axis is the cumulative lower-level response-oracle budget.}
    \label{fig:PWA_stationarity}
\end{figure}

Figure~\ref{fig:PWA_stationarity} shows that the stationarity certificate decreases as the response-oracle budget increases, which is consistent with the convergence analysis and indicates that the iterates approach a $(\delta,\epsilon)$-Bi-ZOL Frank--Wolfe stationary point.
For a batch size $B$, it is worth mentioning that each VZO iteration uses $2B$ response queries, while each Bi-ZOL iteration uses $2B+1$ response queries because it also evaluates $y(x_k)$ at the current point. 
Hence, under the same response-oracle budget, Bi-ZOL may perform fewer iterations than VZO.
The Clarke gap curves provide an additional nonsmooth stationarity check using the exact PWA geometry.
Bi-ZOL also significantly reduces the exact Clarke Frank--Wolfe gap along the trajectory, while VZO remains at a larger Clarke gap under the same budget.
This indicates that the structure-based hypergradient estimation better aligns the true Clarke stationarity geometry of the nonsmooth bilevel problem.
As $B$ increases, all Frank--Wolfe gaps improve further because the variance term caused by the two-point estimator is reduced, leading to more accurate descent directions.

\subsection{Hyperobjective value}
We further compare the true hyperobjective value trajectories of Bi-ZOL and VZO to evaluate the solution quality of the proposed zeroth-order method.

\begin{figure}[htbp]
    \centering
    \includegraphics[width=\columnwidth]{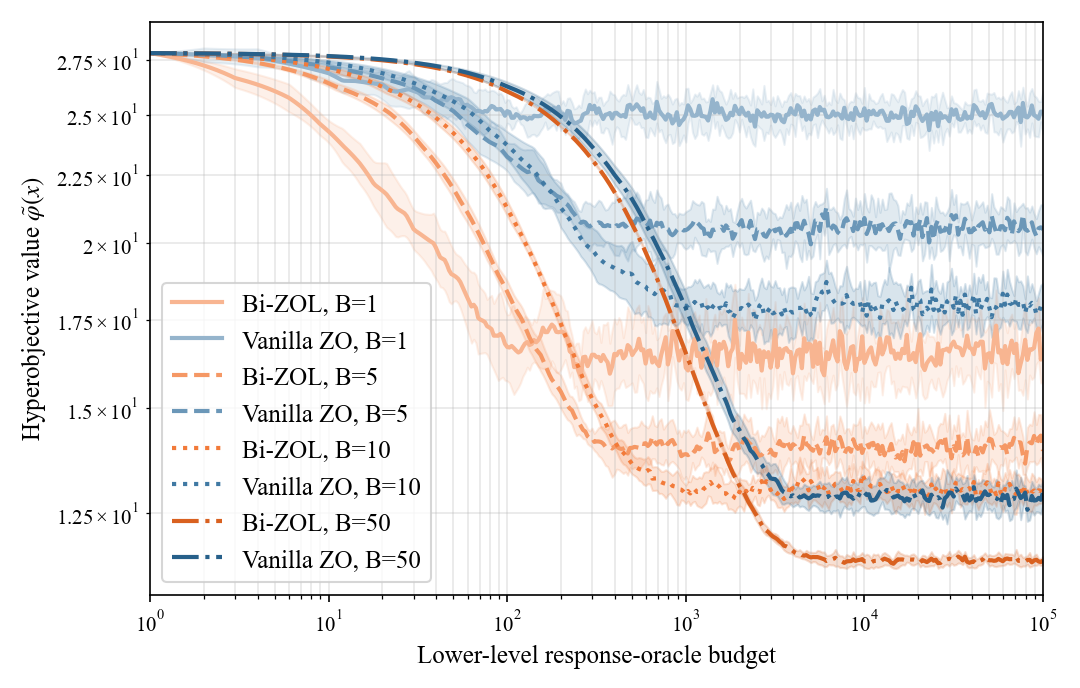}
    \caption{Hyperobjective value trajectories for the PWA response setting with $n=20$ and $m=200$. 
   Across all batch sizes, Bi-ZOL achieves a lower objective value with general reduced variance compared to VZO under the same response-oracle budget.}
    \label{fig:PWA_obj}
\end{figure}

\begin{table}[htb]
    \centering
    \caption{Comparison of objective values of two methods (PWA response map, $n=20$, $m=200$).}
    \label{tab:PWA_results}
    \resizebox{\columnwidth}{!}{%
    \begin{tabular}{l c c c}
        \toprule
        \textbf{Batch size} & \textbf{Bi-ZOL final $\tilde{\varphi}(x)$}\footnotemark & \textbf{VZO final $\tilde{\varphi}(x)$} & \textbf{Relative improvement $\Delta_{\rm obj}$} \footnotemark \\
        \midrule
        B=1  & $16.59 \pm 0.40$ & $24.95 \pm 0.27$ & $33.48\% \pm 1.88\%$ \\
        B=5  & $14.17 \pm 0.60$ & $20.60 \pm 0.80$ & $31.10\% \pm 4.71\%$ \\
        B=10 & $12.99 \pm 0.21$ & $18.04 \pm 0.60$ & $27.93\% \pm 2.94\%$ \\
        B=50 & $11.54 \pm 0.09$ & $12.91 \pm 0.40$ & $10.49\% \pm 2.60\%$ \\
        \bottomrule
    \end{tabular}}
\end{table}

\addtocounter{footnote}{-1}
\footnotetext{To counter the randomness, we compute the mean of last 10 iterations' hyperobjective value as the final hyperobjective value.}

\stepcounter{footnote}
\footnotetext{The relative improvement $\Delta_{\rm obj}$ is defined as
\(
\Delta _{\rm obj}=\frac{\tilde{\varphi}(x)^{\rm final} _{\rm VZO}-\tilde{\varphi}(x)^{\rm final} _{\rm Bi-\rm ZOL}}{| \tilde \varphi(x)^{\rm final} _{\rm VZO}|}.
\)}

Fig~\ref{fig:PWA_obj} reports the hyperobjective value under the same response-oracle budget. 
Across all batch sizes $B=1,5,10, 50$, Bi-ZOL reaches a lower objective value than VZO under the same budget. 
This indicates that the Bi-ZOL direction is more effective in objective minimization in this PWA response setting.
The improvement comes from two factors. 
First, Bi-ZOL uses the exact derivatives of the known upper-level objective at the current pair $(x,y(x))$, so that the target approximate hypergradient $g_\delta(x)$ indicates an efficient update direction.
In contrast, VZO estimates the gradient of a fully smoothed hyperobjective, which is $\nabla \tilde \varphi_\delta(x)$, so its direction may be affected by full-smoothing bias.
Second, Bi-ZOL's zeroth-order sampling is restricted to the response-sensitivity term, rather than the full reduced objective. 
Thus, during the approximation of the surrogate hypergradient, the two-point estimator of Bi-ZOL has less noise.

With the $B$ increase, both methods attain lower objective values, which results from the lower noise of two-point estimators.
Moreover, the gap between Bi-ZOL and VZO decreases as $B$ increases. 
This is expected because larger batches significantly reduce the variance of VZO estimators.
Hence, VZO benefits more from increasing $B$, while the advantage of Bi-ZOL is most visible in the low-to-moderate batch regime, where response-oracle queries are limited. 
Table~\ref{tab:PWA_results} summarizes the fixed-budget performance for different batch sizes, including the final objective value and the relative performance improvement of Bi-ZOL with respect to VZO.
From Table~\ref{tab:PWA_results}, we can also conclude that except for $B=1$, Bi-ZOL is generally more stable across random runs, as indicated by the smaller objective standard deviation.

\begin{figure}[htbp]
    \centering
    \includegraphics[width=\columnwidth]{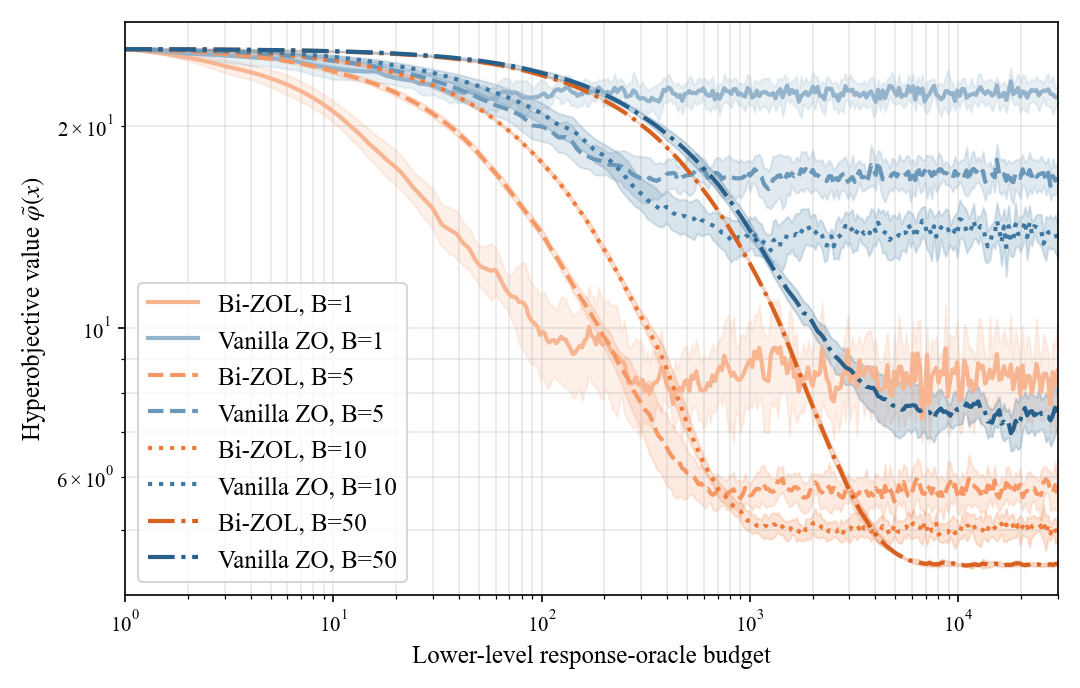}
    \caption{Hyperobjective value trajectories for the non-PWA response setting with $n=20$ and $m=200$.  
    The results demonstrate that Bi-ZOL maintains a lower final objective even when the lower-level response map has a curved active boundary.}
    \label{fig:nonPWA_obj}
\end{figure}

\begin{table}[htbp]
    \centering
    \caption{Comparison of objective values of two methods (non-PWA response map, $n=20$, $m=200$).}
    \label{tab:nonPWA_results}
    \resizebox{\columnwidth}{!}{%
    \begin{tabular}{l c c c c}
        \toprule
        \textbf{\makecell{Batch \\size}} & 
        \textbf{\makecell{Bi-ZOL final \\ $\tilde{\varphi}(x)$}} & 
        \textbf{\makecell{VZO final \\ $\tilde{\varphi}(x)$}} & 
        \textbf{\makecell{Relative improvement \\ $\Delta_{\rm obj}$}} & 
        \textbf{\makecell{Bi-ZOL  \\ $\widehat{\mathcal{G}}^{\mathrm{val}}_{\delta,\mathrm{best}}(x)$}} \\
        \midrule
        B=1  & $8.84 \pm 1.15$ & $22.13 \pm 0.91$ & $59.72\% \pm 5.67\%$ & $1.42 \pm 0.10$ \\
        B=5  & $5.77 \pm 0.24$ & $16.68 \pm 0.75$ & $65.77\% \pm 1.78\%$ & $0.51 \pm 0.09$ \\
        B=10 & $4.96 \pm 0.09$ & $13.43 \pm 0.47$ & $63.23\% \pm 1.43\%$ & $0.19 \pm 0.05$ \\
        B=50 & $4.45 \pm 0.02$ & $7.43 \pm 0.41$ & $40.97\% \pm 3.94\%$ & $0.0034 \pm 0.0016$ \\
        \bottomrule
    \end{tabular}}
\end{table}

We next replace the individual box constraints by the Euclidean ball constraint.
This models an aggregate response budget shared by all lower-level response components.
Unlike the box-constrained case, this response map is not PWA. Its active boundary is curved, which allows us to test whether Bi-ZOL remains effective beyond the PWA setting.
Fig.~\ref{fig:nonPWA_obj} reports the hyperobjective trajectory and Table~\ref{tab:nonPWA_results} gives statistical results.
The results show that Bi-ZOL continues to decrease the objective and its stationarity certificate under the non-PWA response map, and the relative improvement is even larger than in the PWA case.
These results suggest that the benefit of Bi-ZOL is not restricted to PWA response maps.

\subsection{Sensitivity to problem dimensions}

We next examine the scalability of Bi-ZOL and VZO with different problem dimensions.
We use a complete grid
\[
n\in\{20,40,80,100\},
\qquad
m\in\{40,80,100,200\},
\]
with $B=10$.
For each tested dimension pair $(n,m)$, both methods are run from the same
initial point and on the same problem instance under a fixed response-oracle
budget. Since the absolute objective scale may vary with the dimensions, we
report the best-so-far normalized objective gap \footnote{The ${\rm NG}_{\rm obj}^{A}(n,m)$ is a random seed-averaged value.}
\[
{\rm NG}_{\rm obj}^{A}(n,m)
=
\frac{
\Phi_{\rm best}^{A}(n,m)-\Phi_{\rm ref}(n,m)
}{
\Phi_0(n,m)-\Phi_{\rm ref}(n,m)+10^{-8}
},
\]
where $\Phi(n,m)$ denotes the hyperobjective value for the dimension pair $(n,m)$, $A\in\{\mathrm{Bi\text{-}ZOL},\mathrm{VZO}\}$ and
$\Phi_0(n,m),\Phi_{\rm ref}(n,m)$ are the initial and lowest objective value observed over all methods,
random seeds, and iterates for the same dimension pair, respectively. 
Hence smaller values indicate that a method closes a larger portion of the initial objective gap.

\begin{figure}[htbp]
    \centering
    \includegraphics[width=\columnwidth]{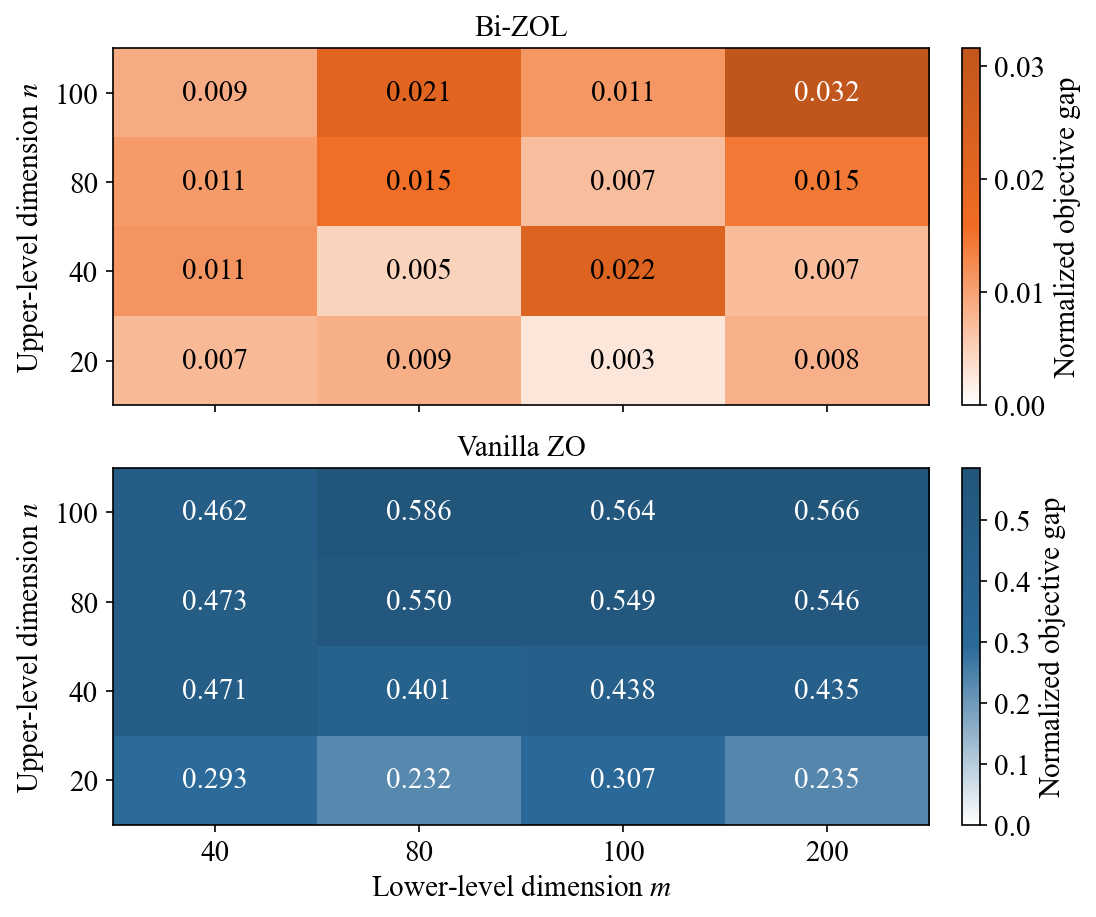}
    \caption{Dimension sensitivity under a fixed response-oracle budget.
    Each entry reports the seed-averaged best-so-far normalized objective gap.
    For each dimension pair \((n,m)\), the reference value is the lowest
    objective value observed over all methods, seeds, and iterates for that
    same dimension pair. Smaller values indicate better performance.}
    \label{fig:dimension_gap_heatmaps}
\end{figure}

Fig.~\ref{fig:dimension_gap_heatmaps} shows that Bi-ZOL consistently attains
a normalized objective gap close to zero over all tested dimension pairs. The
mean normalized gap of Bi-ZOL ranges from \(0.0027\) to \(0.0316\), with an
overall average of \(0.0121\). In contrast, VZO leaves a much larger residual
gap. Its mean normalized gap ranges from \(0.2321\) to \(0.5861\), with an
overall average of \(0.4443\). 
Moreover, for each \((n,m)\),
\(
{\rm NG}_{\rm obj}^{\rm VZO}(n,m)
>
{\rm NG}_{\rm obj}^{\rm Bi\text{-}ZOL}(n,m)
.
\) 
The average difference over the grid is
\(0.4322\), and the largest observed difference is \(0.5651\).
This result shows that Bi-ZOL outperforms VZO on all tested dimension pairs.

We further analyze the sensitivity of results across varying $n$ and $m$, respectively.
Table~\ref{tab:dimension_by_n} summarizes the results after averaging over the response dimension \(m\).
As \(n\) increases from \(20\) to \(100\), the normalized gaps of both methods increase.
Bi-ZOL remains close to the best observed objective value over the tested range of \(n\), whereas VZO leaves a substantially larger residual normalized gap.
This is consistent with the fact that VZO applies zeroth-order smoothing to the whole reduced objective in the \(n\)-dimensional upper-level space, while Bi-ZOL keeps the upper-level partial derivatives exact and applies zeroth-order sampling only to the response-sensitivity term.
This partial-smoothing structure leads to a more stable normalized objective gap over the tested dimensions.

\begin{table}[htp]
\centering
\caption{Normalized gap averaged over lower-level dimensions \(m\).}
\label{tab:dimension_by_n}
\resizebox{\columnwidth}{!}{%
\begin{tabular}{c c c c}
\toprule
\textbf{Upper-level dimension \(n\)}
& \textbf{${\rm NG}_{\rm obj}^{\rm Bi-ZOL}$}
& \textbf{${\rm NG}_{\rm obj}^{\rm VZO}$}
& \textbf{${\rm NG}_{\rm obj}^{\rm VZO} - {\rm NG}_{\rm obj}^{\rm Bi-ZOL}$} \\
\midrule
20  & 0.0068 & 0.2668 & 0.2600 \\
40  & 0.0114 & 0.4365 & 0.4251 \\
80  & 0.0120 & 0.5293 & 0.5173 \\
100 & 0.0182 & 0.5445 & 0.5263 \\
\bottomrule
\end{tabular}}
\end{table}

Table~\ref{tab:dimension_by_m} reports the results after averaging over the upper-level dimension \(n\). 
The Bi-ZOL gap remains close to zero, ranging from \(0.0096\) to \(0.0155\), whereas the VZO gap stays at a much larger level, ranging from \(0.4247\) to \(0.4645\). 
Although the Bi-ZOL gap mildly increases between \(m=40\) and \(m=200\) in terms of the absolute value, the relative variation is larger than that of VZO.
This is consistent with the fact that \(m\) explicitly affects the convergence and oracle complexity of Bi-ZOL as suggested in Theorem~\ref{thm:solution_bizol_convergence} and Corollary~\ref{cor:bizol_response_oracle_complexity}, whereas the standard upper-space zeroth-order analysis for VZO does not explicitly capture this response-dimension dependence in the same way.

\begin{table}[htp]
\centering
\caption{Normalized gap averaged over upper-level dimensions \(n\).}
\label{tab:dimension_by_m}
\resizebox{\columnwidth}{!}{%
\begin{tabular}{c c c c}
\toprule
\textbf{Lower-level dimension \(m\)}
& \textbf{${\rm NG}_{\rm obj}^{\rm Bi-ZOL}$}
& \textbf{${\rm NG}_{\rm obj}^{\rm VZO}$}
& \textbf{${\rm NG}_{\rm obj}^{\rm VZO} - {\rm NG}_{\rm obj}^{\rm Bi-ZOL}$} \\
\midrule
40  & 0.0096 & 0.4247 & 0.4150 \\
80  & 0.0125 & 0.4424 & 0.4299 \\
100 & 0.0108 & 0.4645 & 0.4537 \\
200 & 0.0155 & 0.4455 & 0.4301 \\
\bottomrule
\end{tabular}}
\end{table}

Finally, the comparison is consistent at the run level. 
Across the \(16\) dimension pairs and \(10\) random seeds for each pair, Bi-ZOL attains a smaller normalized objective gap than VZO in all \(160\) paired runs. 
This confirms that the advantage shown in the heatmaps is not caused by a small number of favorable random seeds.

\section{Conclusions}
\label{sec:conclusions}
In this paper, we studied lower-level constrained bilevel optimization in the response-oracle setting, where the lower-level model is unavailable and the response mapping may be nonsmooth. 
We proposed Bi-ZOL, a structure-guided zeroth-order Frank--Wolfe method that estimates only the missing response-sensitivity term while using exact upper-level partial derivatives at the true response. 
This partial-smoothing design connects the computable direction to the bilevel Clarke chain-rule structure and avoids treating the reduced hyperobjective as a generic black-box function.

We established a Bi-ZOL Frank--Wolfe stationarity certificate, quantified its structural bias, and proved finite-time convergence with response-oracle complexity guarantees. 
We also showed how this certificate transfers to Clarke Frank--Wolfe stationarity of the original nonsmooth reduced problem. 
Experiments on incentive-based tracking problems show that Bi-ZOL attains smaller stationarity gaps and lower hyperobjective values than vanilla zeroth-order smoothing, demonstrating the value of using bilevel structure in nonsmooth response-oracle problems.

\appendices
\section{Proofs for Nonsmooth First-Order Geometry}

\subsection{Proof of Proposition~\ref{prop:bl_chain_rule}} \label{app:proof_bl_chain_rule}
Let $\Omega_y$ be the full-measure set on which $y(\cdot)$ is differentiable.
For any sequence $\{x^k\}\subset\Omega_y$ with $x^k\to x$, the classical
chain rule gives
\[
\nabla \tilde\varphi(x^k)
=
\nabla_1\varphi(x^k,y(x^k))
+
\jac y(x^k)^\top
\nabla_2\varphi(x^k,y(x^k)).
\]
Since $y(\cdot)$ is locally Lipschitz and
$\varphi\in C^1(\mathbb R^n\times\mathbb R^m)$, we have
\[
\nabla_1\varphi(x^k,y(x^k))\to d_x(x),
\qquad
\nabla_2\varphi(x^k,y(x^k))\to d_y(x).
\]
Moreover, local Lipschitz continuity of $y(\cdot)$ implies that the
Jacobians $\jac y(x^k)$ are locally bounded. Hence every convergent
subsequence of $\{\jac y(x^k)\}$ has a limit $V\in\clarke y(x)$. Along such
a subsequence,
\[
\nabla \tilde\varphi(x^k)
\to
d_x(x)+V^\top d_y(x).
\]
Thus every limiting gradient of $\tilde\varphi$ at $x$ belongs to the set on
the right-hand side of \eqref{eq:clarke_subdiff_bl_chain_rule}.

Conversely, by the definition of the Clarke generalized Jacobian, every
$V\in\clarke y(x)$ belongs to the convex hull of limiting Jacobians of
$y(\cdot)$ around $x$. Since the mapping
\[
V\mapsto d_x(x)+V^\top d_y(x)
\]
is affine, its image over this convex hull is the convex hull of the
corresponding limiting gradients of $\tilde\varphi$ around $x$. By the
limiting-gradient characterization of the Clarke subdifferential, this gives
\eqref{eq:clarke_subdiff_bl_chain_rule}.

\subsection{Proof of Proposition~\ref{prop:anchored_model_interpretation}} \label{app:proof_anchored_model}

Since $Q_x$ differs from the scalarized response
$u\mapsto\langle d_y(x),y(u)\rangle$ only by a constant and a linear term, the
Clarke chain rule gives
\[
\partial^c Q_x(x)
=
d_x(x)
+
\left\{
V^\top d_y(x):V\in\clarke y(x)
\right\}.
\]
By Proposition~\ref{prop:bl_chain_rule}, the right-hand side is precisely
$\partial^c\tilde\varphi(x)$. This proves
\eqref{eq:anchored_model_same_clarke}.

Next, from the definition of $Q_{x,\delta}$,
\[
\nabla Q_{x,\delta}(u)
=
d_x(x)+\jac y_\delta(u)^\top d_y(x).
\]
Evaluating this identity at $u=x$ gives \eqref{eq:g_delta_gradient_anchored_model}.

It remains to prove the Goldstein inclusion. Observe that $Q_{x,\delta}$ is
the uniform smoothing of $Q_x$. Indeed, using the symmetry of the uniform
distribution on $\mathbb B_n$, we have
\[
\mathbb E_{\xi\sim{\rm Unif}(\mathbb B_n)}
\left[
Q_x(z+\delta \xi)
\right]
=
Q_{x,\delta}(z),
\]
because $\mathbb E[\xi]=0$. Since $Q_x$ is locally Lipschitz, the
randomized smoothing--Goldstein relation of
\cite[Theorem~3.1]{Lin2022} gives
\[
\nabla Q_{x,\delta}(x)\in \partial_\delta Q_x(x).
\]
Together with \eqref{eq:g_delta_gradient_anchored_model}, this yields
\[
g_\delta(x)\in \partial_\delta Q_x(x).
\]

Finally, for any $s$ with $\|s-x\|\le\delta$, the Clarke subdifferential of
$Q_x$ at $s$ satisfies
\[
\partial^c Q_x(s)
=
d_x(x)
+
\left\{
V^\top d_y(x):V\in\clarke y(s)
\right\}.
\]
Taking the convex hull of the union over all $\|s-x\|\le\delta$ gives
\eqref{eq:anchored_goldstein_set}.

\section{Proofs for Bi-ZOL Algorithm}
\label{app:proofs_zo_estimation}

\subsection{Proof of Lemma~\ref{lem:solution_hg_estimator}}
\label{app:proof_bizol_hypergradient_estimator}

We first recall the following standard scalar randomized smoothing estimates,
which will be applied to the anchored scalar function
\(h_x(u)=d_y(x)^\top y(u)\).

\begin{lemma}[Scalar randomized smoothing estimates]
\label{lem:scalar_smoothing_estimates}
Let \(h:\mathcal N\to\mathbb R\) be \(L_h\)-Lipschitz on
\(\mathcal X+\bar\delta\mathbb B_n\). Define
\[
h_{\delta}(u)
\coloneqq
\mathbb E_{\xi\sim{\rm Unif}(\mathbb B_n)}
[h(u+\delta \xi)].
\]
Define the two-point estimator
\[
\apnbla h_{\delta}(u)
\coloneqq
\frac{n}{2\delta}
\left(
h(u+\delta w)-h(u-\delta w)
\right)w ,
\]
where \(w\sim{\rm Unif}(\mathbb S^{n-1})\). Then
\begin{align}
\label{eq:scalar_smoothing_unbiased}
\mathbb E_w[\apnbla h_{\delta}(u)]
&=
\nabla h_{\delta}(u),\\
\label{eq:scalar_smoothing_second_moment}
\mathbb E_w[\|\apnbla h_{\delta}(u)\|^2]
&\le
16\sqrt{2\pi} nL_h^2.
\end{align}
Moreover, \(h_{\delta}\) is differentiable and
\begin{equation}
\label{eq:scalar_smoothing_gradient_lipschitz}
\|\nabla h_{\delta}(u)-\nabla h_{\delta}(u')\|
\le
\frac{cL_h\sqrt{n}}{\delta}\|u-u'\|,
\,
\forall u,u'\in\mathcal X,
\end{equation}
where \(c>0\) is a numerical constant.
\end{lemma}

\begin{proof}
The result follows from randomized smoothing identities; see
\cite[Proposition~2.2, Lemma~D.1]{Lin2022}.
\end{proof}

We now prove Lemma~\ref{lem:solution_hg_estimator}.

\begin{proof}[Proof of Lemma~\ref{lem:solution_hg_estimator}]
Fix \(x\in\mathcal X\) and recall
\[
d_x(x)\coloneqq \nabla_1\varphi(x,y(x)),
\qquad
d_y(x)\coloneqq \nabla_2\varphi(x,y(x)).
\]
By Assumption~\ref{ass:upper_regular}, we have \(\|d_y(x)\|\le M_y\).

Define the scalar function
\[
h_x(u)\coloneqq d_y(x)^\top y(u).
\]
Since \(y(\cdot)\) is \(L_y\)-Lipschitz on the query region,
\[
|h_x(u)-h_x(u')|
\le
\|d_y(x)\|\|y(u)-y(u')\|
\le
M_yL_y\|u-u'\|.
\]
Thus \(h_x\) is \(M_yL_y\)-Lipschitz. For fixed \(x\), we have
\begin{equation}
\label{eq:jac_action_scalar_identity}
\apjac y_\delta(x;w)^\top d_y(x)
=
\frac{n}{2\delta}
\left(
h_x(x+\delta w)-h_x(x-\delta w)
\right)w .
\end{equation}
Thus, the stochastic part of Bi-ZOL can be analyzed through the scalar
two-point estimator in Lemma~\ref{lem:scalar_smoothing_estimates} applied to
\(h_x\).

Combining \eqref{eq:scalar_smoothing_unbiased} with
\eqref{eq:jac_action_scalar_identity}, we have
\[
\begin{aligned}
\mathbb E_w[\apjac y_\delta(x;w)^\top d_y(x)]
&=
\mathbb E_w
\left[
\frac{n}{2\delta}
\left(
h_x(x+\delta w)-h_x(x-\delta w)
\right)w
\right] \\
&=
\nabla h_{x,\delta}(x),
\end{aligned}
\]
where
\[
h_{x,\delta}(u)
\coloneqq
\mathbb E_{\xi\sim{\rm Unif}(\mathbb B_n)}
[h_x(u+\delta \xi)].
\]
Moreover, for any \(u\in\mathcal X\),
\[
h_{x,\delta}(u)
=
\mathbb E_{\xi\sim{\rm Unif}(\mathbb B_n)}
[d_y(x)^\top y(u+\delta \xi)]
=
d_y(x)^\top y_\delta(u).
\]
Therefore,
\[
\nabla h_{x,\delta}(u)
=
\jac y_\delta(u)^\top d_y(x).
\]
Taking \(u=x\) gives
\[
\nabla h_{x,\delta}(x)
=
\jac y_\delta(x)^\top d_y(x).
\]
Therefore,
\[
\mathbb E_w[\widehat g_\delta(x;w)]
=
d_x(x)+\jac y_\delta(x)^\top d_y(x)
=
g_\delta(x).
\]
Averaging \(B\) independent samples gives
\eqref{eq:hg_estimator_unbiased}.

For the variance bound, let
\[
\zeta_b
\coloneqq
\apjac y_\delta(x;w_b)^\top d_y(x)
-
\jac y_\delta(x)^\top d_y(x).
\]
Then
\[
\widehat g_{\delta,B}(x)-g_\delta(x)
=
\frac{1}{B}\sum_{b=1}^B \zeta_b.
\]
The random vectors \(\{\zeta_b\}_{b=1}^B\) are independent and have zero mean.
Therefore,
\[
\mathbb E
\left[
\left\|
\widehat g_{\delta,B}(x)-g_\delta(x)
\right\|^2
\right]
=
\frac{1}{B^2}
\sum_{b=1}^B
\mathbb E[\|\zeta_b\|^2]
=
\frac{1}{B}
\mathbb E[\|\zeta_1\|^2].
\]
Using
\[
\mathbb E[\|\zeta_1\|^2]
\le
\mathbb E
\left[
\left\|
\apjac y_\delta(x;w_1)^\top d_y(x)
\right\|^2
\right],
\]
and applying \eqref{eq:scalar_smoothing_second_moment} to the
\(M_yL_y\)-Lipschitz function \(h_x\), we obtain
\[
\mathbb E[\|\zeta_1\|^2]
\le
16\sqrt{2\pi}nM_y^2L_y^2.
\]
This proves \eqref{eq:hg_estimator_variance}.
\end{proof}

\subsection{Proof of Lemma~\ref{lem:solution_ancillary_smoothness}} \label{app:proof_solution_ancillary_smoothness}
For \(x\in\mathcal X\),
\[
\nabla\bar\varphi_\delta(x)
=
\nabla_1\varphi(x,y_\delta(x))
+
\jac y_\delta(x)^\top
\nabla_2\varphi(x,y_\delta(x)).
\]
Let \(x,x'\in\mathcal X\). For the first term of $\nabla\bar\varphi_\delta(x)$, Assumption
\ref{ass:upper_regular} gives
\[
\begin{aligned}
&
\left\|
\nabla_1\varphi(x,y_\delta(x))
-
\nabla_1\varphi(x',y_\delta(x'))
\right\|\\
&\le
L_{1x}\|x-x'\|
+
L_{1y}\|y_\delta(x)-y_\delta(x')\|.
\end{aligned}
\]
Since \(y_\delta\) is \(L_y\)-Lipschitz, we have
\begin{equation}
\label{eq:proof_smooth_first_part}
\left\|
\nabla_1\varphi(x,y_\delta(x))
-
\nabla_1\varphi(x',y_\delta(x'))
\right\|
\le
(L_{1x}+L_{1y}L_y)\|x-x'\|.
\end{equation}

For the second term of $\nabla\bar\varphi_\delta(x)$, define
\[
v_\delta(x)
\coloneqq
\nabla_2\varphi(x,y_\delta(x)).
\]
Then
\begin{multline}
\left\| \jac y_\delta(x)^\top v_\delta(x) - \jac y_\delta(x')^\top v_\delta(x') \right\| \\
\le \left\| \jac y_\delta(x)^\top \left( v_\delta(x)-v_\delta(x') \right) \right\| \\
+ \left\| \left( \jac y_\delta(x)-\jac y_\delta(x') \right)^\top v_\delta(x') \right\|.
\end{multline}
For the first term,
\[
\left\|
\jac y_\delta(x)^\top
\left(
v_\delta(x)-v_\delta(x')
\right)
\right\|
\le
M_J\|v_\delta(x)-v_\delta(x')\|.
\]
By Assumption~\ref{ass:upper_regular},
\[
\begin{aligned}
\|v_\delta(x)-v_\delta(x')\|
&=
\left\|
\nabla_2\varphi(x,y_\delta(x))
-
\nabla_2\varphi(x',y_\delta(x'))
\right\|\\
&\le
L_{2x}\|x-x'\|
+
L_{2y}\|y_\delta(x)-y_\delta(x')\|\\
&\le
(L_{2x}+L_{2y}L_y)\|x-x'\|.
\end{aligned}
\]
Thus,
\begin{equation}
\label{eq:proof_smooth_second_part_a}
\left\|
\jac y_\delta(x)^\top
\left(
v_\delta(x)-v_\delta(x')
\right)
\right\|
\le
M_J(L_{2x}+L_{2y}L_y)\|x-x'\|.
\end{equation}

For the second term, define $h_{v}(u)=v_\delta(x')^\top y(u)$. 
We use the bound \(\|\nabla_2\varphi(x,y_\delta(x))\|\le M_y\), which is included in Assumption~\ref{ass:upper_regular}.
Because Assumption~\ref{ass:response} and \ref{ass:upper_regular}, it is easy to prove that $h_{v}(\cdot)$ is $M_yL_y$-Lipschitz.
Moreover,
\[
h_{v,\delta}(u)
=
\mathbb E[h_{v}(u+\delta w)]
=
v_\delta(x')^\top y_\delta(u),
\]
so
\[
\nabla h_{v,\delta}(u)
=
\jac y_\delta(u)^\top v_\delta(x').
\]
Applying \eqref{eq:scalar_smoothing_gradient_lipschitz} to \(h_v\), we get
\begin{flalign}
&\begin{aligned}
\left\| \left( \jac y_\delta(x)-\jac y_\delta(x') \right)^\top v_\delta(x') \right\| &= \|\nabla h_{v,\delta}(x)-\nabla h_{v,\delta}(x')\| \\
&\le \frac{cM_yL_y\sqrt{n}}{\delta}\|x-x'\|.
\end{aligned} &&
\end{flalign}
Combining this estimate with
\eqref{eq:proof_smooth_first_part} and
\eqref{eq:proof_smooth_second_part_a} gives
\[
\|\nabla\bar\varphi_\delta(x)-\nabla\bar\varphi_\delta(x')\|
\le
\left(
A_y+\frac{cM_yL_y\sqrt{n}}{\delta}
\right)
\|x-x'\|.
\]
This proves the lemma.

\subsection{Proof of Lemma~\ref{lem:solution_bias}} \label{app:proof_solution_bias}
By definition,
\[
\begin{aligned}
&\nabla\bar\varphi_\delta(x)-g_\delta(x)
=
\nabla_1\varphi(x,y_\delta(x))
-
\nabla_1\varphi(x,y(x))\\
&\quad+
\jac y_\delta(x)^\top
\left[
\nabla_2\varphi(x,y_\delta(x))
-
\nabla_2\varphi(x,y(x))
\right].
\end{aligned}
\]
Taking norms and using the triangle inequality gives
\[
\begin{aligned}
&\|\nabla\bar\varphi_\delta(x)-g_\delta(x)\|
\le
\left\|
\nabla_1\varphi(x,y_\delta(x))
-
\nabla_1\varphi(x,y(x))
\right\|\\
&\quad+
\|\jac y_\delta(x)\|
\left\|
\nabla_2\varphi(x,y_\delta(x))
-
\nabla_2\varphi(x,y(x))
\right\|.
\end{aligned}
\]
By Assumption~\ref{ass:upper_regular} and
\(\|\jac y_\delta(x)\|\le M_J\), we obtain
\[
\|\nabla\bar\varphi_\delta(x)-g_\delta(x)\|
\le
(L_{1y}+M_JL_{2y})\|y_\delta(x)-y(x)\|.
\]
Finally,
\[
\begin{aligned}
\|y_\delta(x)-y(x)\|
&=
\left\|
\mathbb E_{\xi\sim{\rm Unif}(\mathbb B_n)}
[y(x+\delta \xi)-y(x)]
\right\|\\
&\le
\mathbb E
\left[
\|y(x+\delta \xi)-y(x)\|
\right]\\
&\le
\mathbb E[L_y\delta\|\xi\|]
\le
L_y\delta .
\end{aligned}
\]
Therefore,
\[
\|\nabla\bar\varphi_\delta(x)-g_\delta(x)\|
\le
(L_{1y}+M_JL_{2y})L_y\delta
=
C_{\rm bias}\delta .
\]
Multiplying by \(D\) gives the bound
\(\kappa_{\rm bias}\delta\) for the corresponding Frank--Wolfe inner-product
term.

\subsection{Proof of Theorem~\ref{thm:solution_bizol_convergence}} \label{app:proof_solution_bizol_convergence}
Let $d_k\coloneqq z_k-x_k$.
Since \(z_k,x_k\in\mathcal X\), we have $\|d_k\|\le D$.
We can bound the descent of $\bar\varphi_\delta$ as follows:
\begin{equation}
\label{eq:proof_descent_with_bias}
\begin{aligned}
    \bar\varphi_\delta(&x_{k+1})
    \overset{(s.1)}{\le} \bar\varphi_\delta(x_k) + \gamma \langle \nabla\bar\varphi_\delta(x_k),d_k \rangle + \frac{L_{\bar\varphi,\delta}}{2}\gamma^2D^2 \\
    &\overset{(s.2)}{=} \bar\varphi_\delta(x_k) + \gamma \langle g_\delta(x_k),d_k \rangle \\
    & \qquad + \gamma \langle \nabla\bar\varphi_\delta(x_k)-g_\delta(x_k),d_k \rangle + \frac{L_{\bar\varphi,\delta}}{2}\gamma^2D^2 \\
    &\overset{(s.3)}{\le} \bar\varphi_\delta(x_k) + \gamma \langle g_\delta(x_k),d_k \rangle \\
    & \qquad + \gamma \|\nabla\bar\varphi_\delta(x_k)-g_\delta(x_k)\|\|d_k\| + \frac{L_{\bar\varphi,\delta}}{2}\gamma^2D^2 \\
    &\overset{(s.4)}{\le} \bar\varphi_\delta(x_k) + \gamma \langle g_\delta(x_k),d_k \rangle + \gamma\kappa_{\rm bias}\delta + \frac{L_{\bar\varphi,\delta}}{2}\gamma^2D^2. 
\end{aligned}
\end{equation}
The (s.1) follows from the smoothness property in Lemma~\ref{lem:solution_ancillary_smoothness}.
The (s.2) is obtained by adding and subtracting $\gamma \langle g_\delta(x_k), d_k \rangle$.
The (s.3) is due to the Cauchy-Schwarz inequality.
The (s.4) follows from the upper bound of the bias in Lemma~\ref{lem:solution_bias} and the condition $\|d_k\| \le D$.

Let
\[
z_k^\delta
\in
\arg\max_{z\in\mathcal X}
\left\langle z-x_k,-g_\delta(x_k)\right\rangle .
\]
Then
\[
\mathcal G_\delta(x_k)
=
\left\langle z_k^\delta-x_k,-g_\delta(x_k)\right\rangle .
\]
Because \(z_k\) is the Frank--Wolfe oracle for \(\widehat g_k\),
\[
\left\langle z_k-x_k,-\widehat g_k\right\rangle
\ge
\left\langle z_k^\delta-x_k,-\widehat g_k\right\rangle .
\]
Equivalently,
\begin{equation}
\label{eq:proof_fw_oracle_relation}
\left\langle \widehat g_k,z_k-z_k^\delta\right\rangle
\le 0.
\end{equation}
Now decompose
\begin{flalign}
\label{eq:proof_oracle_error_bound}
&\left\langle g_\delta(x_k),d_k\right\rangle
=
\left\langle g_\delta(x_k),z_k-x_k\right\rangle && \nonumber \\
&=
\left\langle g_\delta(x_k),z_k^\delta-x_k\right\rangle
+
\left\langle g_\delta(x_k),z_k-z_k^\delta\right\rangle && \nonumber \\
&=
-\mathcal G_\delta(x_k)
+
\left\langle g_\delta(x_k)-\widehat g_k,z_k-z_k^\delta\right\rangle + \left\langle \widehat g_k,z_k-z_k^\delta\right\rangle && \nonumber \\
&\leq-\mathcal G_\delta(x_k)
+
D\|g_\delta(x_k)-\widehat g_k\|. &&
\end{flalign}
Substituting \eqref{eq:proof_oracle_error_bound} into
\eqref{eq:proof_descent_with_bias} gives
\begin{align*}
\bar\varphi_\delta(x_{k+1})
\le
\bar\varphi_\delta(x_k)
-
\gamma\mathcal G_\delta(x_k)
+
\gamma D\|g_\delta(x_k)-\widehat g_k\| \\
+ 
\gamma\kappa_{\rm bias}\delta
+
\frac{L_{\bar\varphi,\delta}}{2}\gamma^2D^2 .
\end{align*}
Rearranging,
\begin{align*}
\mathcal G_\delta(x_k)
\le
\frac{\bar\varphi_\delta(x_k)-\bar\varphi_\delta(x_{k+1})}{\gamma}
+
D\|g_\delta(x_k)-\widehat g_k\| \\
+
\kappa_{\rm bias}\delta
+
\frac{L_{\bar\varphi,\delta}}{2}\gamma D^2 .
\end{align*}

Take conditional expectation with respect to the randomness at iteration
\(k\). By Lemma~\ref{lem:solution_hg_estimator},
\[
\mathbb E
\left[
\|g_\delta(x_k)-\widehat g_k\|
\,\middle|\,
x_k
\right]
\le
\sqrt{
\mathbb E
\left[
\|g_\delta(x_k)-\widehat g_k\|^2
\,\middle|\,
x_k
\right]
}
\le
\frac{\sigma_\delta}{\sqrt B}.
\]
Therefore,
\begin{align*}
\mathbb E[\mathcal G_\delta(x_k)]
\le
\frac{
\mathbb E[\bar\varphi_\delta(x_k)]
-
\mathbb E[\bar\varphi_\delta(x_{k+1})]
}{\gamma}
+
\frac{D\sigma_\delta}{\sqrt B} \\
+
\kappa_{\rm bias}\delta
+
\frac{L_{\bar\varphi,\delta}}{2}\gamma D^2 .
\end{align*}
Summing this inequality from \(k=0\) to \(T-1\) gives
\begin{align*}
\sum_{k=0}^{T-1}
\mathbb E[\mathcal G_\delta(x_k)]
\le
\frac{
\bar\varphi_\delta(x_0)
-
\mathbb E[\bar\varphi_\delta(x_T)]
}{\gamma} \\
+
T\left(
\frac{D\sigma_\delta}{\sqrt B}
+
\kappa_{\rm bias}\delta
+
\frac{L_{\bar\varphi,\delta}}{2}\gamma D^2
\right).
\end{align*}
Since
\[
\bar\varphi_\delta(x_T)
\ge
\min_{x\in\mathcal X}\bar\varphi_\delta(x),
\]
we have
\[
\bar\varphi_\delta(x_0)
-
\mathbb E[\bar\varphi_\delta(x_T)]
\le
\Delta_\delta.
\]
Dividing by \(T\) proves \eqref{eq:solution_main_convergence_bound}.

\section*{ACKNOWLEDGMENT}

\bibliographystyle{IEEEtran}
\bibliography{root}

\end{document}